\documentclass{article}
\usepackage{authblk}
\usepackage[utf8]{inputenc}
\usepackage{url}
\usepackage[version=4]{mhchem}
\usepackage{multirow}
\usepackage{lipsum}
\usepackage{stackrel}
\usepackage{listings}
\usepackage{amssymb}
\usepackage{nccmath}
\usepackage{a4wide}
\usepackage{mathtools}
\usepackage{amsmath,amsfonts,amsthm,amssymb}
\usepackage{bbm, dsfont}
\usepackage[justification=justified, singlelinecheck=false]{caption}
\usepackage{graphicx}
\usepackage{enumitem}
\usepackage[justification=centering]{caption}
\usepackage{enumitem}
\usepackage{empheq,etoolbox}
\patchcmd{\subequations}
  {\theparentequation\alph{equation}}
  {\theparentequation.\alph{equation}}
  {}{}
\usepackage[justification=centering]{caption}
\usepackage{algorithm}
\usepackage{algpseudocode}

\usepackage{stackengine}

\usepackage[title]{appendix}
\usepackage[scaled=.90]{helvet}
\usepackage{courier}
\usepackage{ae}
\usepackage[T1]{fontenc}
\usepackage{graphicx}
\usepackage{subcaption}
\usepackage[export]{adjustbox}
\usepackage{wrapfig}

\usepackage[T1]{fontenc}
\usepackage{here} 
\usepackage[colorlinks=false, pdfborder={0 0 0}]{hyperref}

\usepackage{mdwlist}
\usepackage{tcolorbox}
\usepackage{fancybox}
\usepackage{framed}

\definecolor{darkblue}{rgb}{0,0,0.8}
\definecolor{darkgreen}{rgb}{0,0.8,0}

\definecolor{magenta}{rgb}{0.5,0,0.5}

\newcommand{\mathleft}{\@fleqntrue\@mathmargin0pt}

\newtheorem{theorem}{Theorem}[section]
\newtheorem{definition}{Definition}[section]
\newtheorem{lemma}{Lemma}[section]
\newtheorem{Lemma}{Lemma}
\newtheorem{assumption}{Assumption}[section]

\newtheorem{remark}{Remark}[section]

\newtheorem{notation}{Notation}[section]

\providecommand{\keywords}[1]
{
  \small	
  \textbf{\textit{Keywords---}} #1
}

\newcommand{\FTS}[2]{\frac{{\textstyle #1}}{{\textstyle #2}}}

\usepackage{biblatex}
\begin{document}
\title{Numerical Approximations and Convergence Analysis of Piecewise Diffusion Markov Processes, with Application to Glioma Cell Migration} 
\author[1]{Evelyn Buckwar\thanks{Email: evelyn.buckwar@jku.at}}
\author[1]{Amira Meddah\thanks{Corresponding author: amira.meddah@jku.at}}
\affil[1]{\centerline{\small Institute of Stochastics, Johannes Kepler University Linz,} \newline \centerline{\small Altenberger Straße 69, 4040 Linz, Austria}}
\date{\today}

\newcommand{\ebbox}[1]{\fbox{$\triangleright$\textcolor{green}{\textbf{Evelyn}:} #1}}
\newcommand{\ebbbox}[1]{{
\fbox{
\parbox{0.9\textwidth}{  \fbox{$\triangleright$\asd{\textbf{Evelyn}:}} 
#1
}}}}

\newcommand{\amebox}[1]{\fbox{$\triangleright$\textcolor{blue}{\textbf{Amira}:} #1}}
\newcommand{\amebbox}[1]{{
\fbox{
\parbox{0.9\textwidth}{  \fbox{$\triangleright$\asd{\textbf{Amira}:}} 
#1
}}}}

\maketitle
\begin{abstract}
In this paper, we focus on numerical approximations of Piecewise Diffusion Markov Processes (PDifMPs), particularly when the explicit flow maps are unavailable. Our approach is based on the thinning method for modelling the jump mechanism and combines the Euler-Maruyama scheme to approximate the underlying flow dynamics. For the proposed approximation schemes, we study both the mean-square and weak convergence. Weak convergence of the algorithms is established by a martingale problem formulation. Moreover, we employ these results to simulate the migration patterns exhibited by moving glioma cells at the microscopic level. Further, we develop and implement a splitting method for this PDifMP model and employ both the Thinned Euler-Maruyama and the splitting scheme in our simulation example, allowing us to compare both methods.
\end{abstract}\vspace{0.05cm}





\section{Introduction}
\label{section1}
Over the past decades, stochastic hybrid systems (SHS) have emerged as a powerful modelling technique in a wide range of application areas such as mathematical biology \cite{berg1972chemotaxis,cloez2017probabilistic}, neuroscience \cite{pakdaman2010fluid,buckwar2011exact}, biochemistry \cite{singh2010stochastic}, finance \cite{ishijima2011regime}, and power systems \cite{malhame1990jump}, to name a few.\\
A well-known and very powerful class of SHS, characterised by continuous dynamics and stochastic jumps, is the piecewise deterministic Markov process (PDMP), introduced in $1984$ by Davis \cite{davis1984piecewise}. Davis's pioneering work focused on endowing PDMPs with general tools, similar to the already existing one for diffusion processes. Essentially, PDMPs form a family of non-diffusive c\`adl\`ag Markov processes, involving a deterministic motion punctuated by random jumps.\\
\indent A few years later, a very general formal model for stochastic hybrid systems was proposed by Blom in \cite{blom1988piecewise}. This model expanded upon the foundational principles of the PDMPs by substituting the ordinary differential equations (ODEs), which govern continuous dynamics, with their stochastic counterparts. Moreover, the author introduced generalised reset maps that incorporated state-dependent distributions, characterising the probability density of the system state after each transition. This class of newly formulated hybrid systems was named Piecewise Diffusion Markov processes (PDifMPs).\\
Subsequently, in 2003, a follow-up work by Bujorianu et al. \cite{bujorianu2003reachability} rigorously established the theory underlying PDifMPs and defined their extended generators.\\
\indent While PDifMPs are amenable to analytical solutions in certain cases, numerical approximation techniques are often required in practice. One of the main challenges in numerically approximating PDifMPs lies in handling stochastic jumps, which occur at random times and exert a substantial influence on the system dynamics. Furthermore, it is essential to efficiently combine these jumps with the discretisation of the continuous dynamics. Numerous publications have addressed numerical approximations for PDMPs, e.g. \cite{lemaire2018exact, lemaire2020thinning, riedler2013almost, kritzer2019approximation, Bertazzietal2022}. In particular, \cite{Bertazzietal2022} provides a thorough overview of issues arising, when approximating all characteristics of PDMPs with a focus on simulating the underlying distributions. However, there appears to be a gap in the literature regarding the development and analysis of numerical schemes for PDifMPs. Although, simulations of generalised stochastic hybrid systems are presented in \cite{blom2018interacting}, these are limited to specific examples.\\
\indent To address this issue, we propose two distinct approaches that combine different techniques for approximating both the continuous and discrete aspects of the process. Inspired by the work of Lemaire et al. \cite{lemaire2018exact, lemaire2020thinning}, we introduce a first approach for approximating the flow maps of PDifMPs. This approach leverages a combination of the thinning method and the Euler-Maruyama scheme. Additionally, we provide estimates for the mean-square error and introduce a framework for a weak error expansion.\\
Furthermore, we propose a second numerical method that integrates the thinning method with the splitting method to approximate the continuous dynamics of the process between jump times.\\
The fundamental idea behind the splitting method is to decompose the stochastic differential equations (SDEs) governing the continuous component of the process into exactly solvable subequations, derive their solutions and compose them in a suitable way. 
We refer to
\cite{mclachlan2002splitting, casas2008splitting} for a thorough discussion of splitting methods for ODEs and to \cite{buckwar2022splitting, ableidinger2017stochastic, alamo2016technique,shardlow2003splitting, petersen1998general} among many others, for references considering extensions to SDEs. While to the best of our knowledge this approach has not yet been explored in the context of PDifMPs, it has recently been applied to PDMPs by Bertazzi et al. \cite{bertazzi2023piecewise} where the authors introduced novel Markov chain Monte Carlo (MCMC) algorithms that leverage splitting schemes to approximate PDMPs.\\
\indent As an application we consider a PDifMP describing glioma cell migration at a microscopic scale, simplifying the model to one dimension for clarity of exposition. Gliomas, the most prevalent type of intracranial tumour, account for over $70\%$ of all brain cancers. These tumours arise from mutated glial cells, which are a type of cell that provides physical and chemical support to neurons \cite{ohgaki2009epidemiology}. For an extensive discussion on low-grade gliomas and various mathematical models for studying brain tumours we refer to our work \cite{meddah2023stochastic} and the references therein.\\
\indent The paper is organised as follows. In Section \ref{section2}, we provide a concise introduction to PDifMPs. Section \ref{section3} introduces the Thinned Euler-Maruyama (TEM) method for PDifMPs. Sections \ref{section4} and \ref{section5} are dedicated to the numerical analysis, exploring both mean-square and weak convergence. Section \ref{section6} presents an extended model from \cite{meddah2023stochastic}, considering microenvironmental effects on cellular movement, and includes numerical simulations using the TEM method, comparing it with the splitting method. These simulations focus on the impact of parameters like the jump intensity of PDifMPs and the time step size. Section \ref{section7} concludes with a review of our findings and discusses potential future research directions.

\section{Preliminaries}
\label{section2}
In this section, we present the preliminary concepts necessary to understand PDifMPs. For more detailed information, we refer the reader to \cite{bujorianu2006toward, bujorianu2003reachability, bect2007processus}.\\
\indent We consider a complete probability space 
$\left(\Omega, \mathcal{F}, (\mathcal{F})_{t \geq 0}, \mathbb{P}\right)$ with filtration $(\mathcal{F})_{t \geq 0}$ on $\Omega$ satisfying the usual conditions, i.e $(\mathcal{F})_{t \geq 0}$ is right continuous and $\mathcal{F}_0$ contains all $\mathbb{P}$-null sets in $\mathcal{F}$. Let $T$ be a fixed deterministic time, we consider $(W_t)_{t\in [0,T]}\in \mathbb{R}^{m}$, $m\in \mathbb{N}$, to be a standard Wiener process defined on $\left(\Omega, \mathcal{F}, (\mathcal{F})_{t \geq 0}, \mathbb{P}\right)$. We assume that $W_t$ is $\mathcal{F}_t$-adapted.\\
The PDifMP we consider consists of two non-exploding processes $(y_t, v_t)$, with values in $E=E_1\times \mathbf{V}$. Here, $E_1$ and $\mathbf{V}$ are subsets of $\mathbb{R}^{d}$, $d \in \mathbb{N}$, such that $\mathbf{V}$ is finite or a countable set. The Borel $\sigma$-algebra $\mathcal{B}(E)$ is generated by the topology of $E$.\\
\indent The first stochastic component $(y_t)_{t\in [0,T]}$ possesses continuous paths in $E_1$, while the second component $(v_t)_{t\in [0,T]}$ is a jump process with right-continuous paths and piecewise constant values in $\mathbf{V}$. Additionally, we denote by $(T_n)_{n\in \mathbb{N}}$ the times at which the second component jumps. These times form a sequence of randomly distributed grid points in the interval $[0,T]$.\\
\indent The dynamics of the PDifMP $(x_t)=(y_t,v_t)$, $t\in [0,T]$, on $(E,\mathcal{B}(E))$ are defined by its characteristic triple $(\phi, \lambda, \mathcal{Q})$.
The parameter $\phi: [0,T] \times E \rightarrow E_1$, given by $(t,x_t)\mapsto \phi (t,x_t)$, represents the stochastic flow map of the continuous component $(y_t)_{t\in [0,T]}$. It encapsulates the solution to a sequence of SDEs over consecutive intervals $[T_n, T_{n+1})$ with random lengths. At each random point $T_n\in [0,T]$, $n\geq 1$, there are updated initial values $x_n=(y_n,v_n)\in E$. Here, $y_n$ serves as the initial value, and $v_n$ acts as a parameter in SDE \eqref{sys_1}, stated below, defined on the interval $[T_n, T_{n+1})$.\\
For $p\geq 1$, let $\mathcal{L}_{\text{ad}}\left(\Omega, L^{p}[T_n, T_{n+1}]\right)$ denote the class of random processes that are $\mathcal{F}_t$-adapted and exhibit almost surely $p$-integrable sample paths. In the context of Equation \eqref{sys_1}, we define the drift vector $b(\cdot, v_n)$ mapping $E$ into $\mathbb{R}^{d}$ as an element of $\mathcal{L}_{\text{ad}}\left(\Omega, L^{1}[T_n, T_{n+1}]\right)$. Further, the diffusion matrix $\sigma(\cdot, v_n)$, mapping $E$ into $\mathbb{R}^{{d}\times m}$ is a class $\mathcal{L}_{ad}\left(\Omega, L^{2}[T_n, T_{n+1}]\right)$ function. Let 

	\begin{equation}
		\label{sys_1}
		\left\{
		\begin{array}{ll}
			d\phi(t,x_n)=b(\phi(t,x_n),v_n)dt+\sigma(\phi(t,x_n),v_n)dW_t, \qquad t \in [T_n, T_{n+1}),\\[0.2cm]
			\phi(T_n,x_n)=y_n.
		\end{array}
		\right. 
	\end{equation}
	At the end point $T_{n+1}$ of each interval, $y_{n+1}$ is set to the current value of $\phi(\,\cdot\,, x_n)$ to ensure the continuity of the path. Further, a new value $v_{n+1}$ is chosen as fixed parameter for the next interval according to the jump mechanism described below. 
	Equation (\ref{sys_1}) can be formulated equivalently as
	
	\begin{equation}
	    \label{integral_form}
	    \phi(t,x_n)= y_n+ \int_{T_n}^{t}b(\phi(s,x_n),v_n)ds+\int_{T_n}^{t}\sigma(\phi(s,x_n),v_n)dW_{s}, \qquad t \in [T_n, T_{n+1}).
	\end{equation}
	For ease of understanding, we adopt the notation $x_t=(y_t,v_t)$ to represent the sample paths of the PDifMP, with $x=(y,v)$ denoting the initial conditions following each jump. Henceforth, unless otherwise specified, $x=(y,v)$ consistently indicates the starting point for each interval, irrespective of the specific time point under consideration.\\
 The following theorem establishes the existence of a unique solution for Equation (\ref{integral_form}) in $E$.
\begin{theorem}[Mao \cite{mao2007stochastic}]
	\label{existence}
	Assume that there exist two positive constants $K_1$ and $K_2$, such that 
	\begin{enumerate}
	    \item (Lipschitz condition) for all $\varphi, \Tilde{\varphi} \in E_1$ and $v\in \mathbf{V}$
	    \begin{equation}
	        \label{lipschitz_cond}
	        \lvert b(\varphi,v)-b(\Tilde{\varphi},v) \rvert ^2\bigvee \lvert \sigma(\varphi,v)-\sigma(\Tilde{\varphi},v)\rvert ^2
\leq     K_1 \lvert \varphi-\Tilde{\varphi} \rvert ^2,
	    \end{equation}
	    \item (Linear growth condition) for all $(\varphi,v)\in E$
	    \begin{equation}
	    \label{linear_growth}
	        \lvert b(\varphi,v) \rvert ^2 \bigvee \lvert \sigma(\varphi,v) \rvert ^2 \leq K_2 (1+\vert\varphi\vert^2).
	    \end{equation}
	    \end{enumerate}
	    Then, for any $v \in \mathbf{V}$, there exists a unique solution to Equation (\ref{integral_form}) in $E$.
	\end{theorem}
\noindent The stochastic flow $\phi$ possesses the semi-group property
	\begin{equation}
     \label{semi_gr_pr}
		\phi\big((t+s),(\,\cdot\,,v)\big)=\phi\big(t,\, (\phi(s,(\,\cdot\,,v)), v)\big)\,,\quad \forall \, t, \,s \in [0,T], \quad \text{a.s}.
	\end{equation}
\noindent The jump mechanism of the PDifMP is governed by the remaining components of its characteristic triple, namely $\lambda$ and $\mathcal{Q}$. Specifically, the function ${\lambda: E \rightarrow \mathbb{R}_{+}}$ quantifies the rate at which jumps take place in the second component of the process $(x_t)_{t\in [0,T]}$, which in turn determines the sequence of random times $(T_n)_{n>0}$.
\begin{assumption}
	\label{assump_lambda}
Let	$\lambda:E\rightarrow \mathbb{R}_{+}$ be a measurable function. For all $\omega\in \Omega$, $\varphi\in E_1$, $v\in \mathbf{V}$, and $T>0$, we assume that
	\begin{equation}
		\int_0^T \lambda(\varphi,v) dt < \infty \hspace{1cm} \text{and} \hspace{1cm} \int_0^{\infty} \lambda(\varphi,v)  dt =\infty. 
	\end{equation}
\end{assumption}
\noindent The first condition in Assumption \ref{assump_lambda} ensures that the total jump rate over any finite time interval is bounded, preventing excessively frequent jumps. The second condition ensures that the total jump rate over infinite time intervals is unbounded, allowing for a nontrivial occurrence of jumps over time.\\
\noindent Additionally, the transition kernel $\mathcal{Q}: (E,\mathcal{B}(E))\rightarrow [0,1]$ determines the new value of the second component following a jump event. It defines the probability distribution that governs the transition from the pre-jump state to the post-jump state.
\begin{assumption}
For all $x \in \overline{E}$, where $\overline{E}$ denotes the closure of the set $E$, the function $\mathcal{Q}(x,\,\cdot\,)$ is a probability measure. Moreover, for all $A \in \mathcal{B}(E)$, $\mathcal{Q}(\,\cdot\,,A)$ is measurable. Further,
\begin{itemize}
    \item[a.] For any $x \in E$, the condition $\mathcal{Q}(x, \{x\}) = 0$ holds, effectively preventing the process from having a no-move jump, thereby ensuring that each jump results in a state change.
    \item[b.] At any jump time $T_n$ of $v_t$, where $v_{T_n^-} \neq v_{T_n}$ and $T_n^-$ denotes the time just before $T_n$, the continuity of $y_t$ is preserved in the sense that $y_{T_n} = y_{T_{n}^{-}}$.
\end{itemize}
\label{Trans_Assump}
\end{assumption}

\begin{definition}
 For any $t\in [T_n, T]$ with $n\geq 0$, and $\phi (t,x)$ the solution of Equation \eqref{sys_1}, the \textit{survival function of the inter-jump times} $\mathcal{S}(t,x)$ is defined as follows
\begin{equation}
 \label{gen_surv}
 \mathcal{S}(t,x):= \exp \left(-\int_{T_n}^t \lambda(\phi(s,x),v)ds \right), \qquad x\in E.
 \end{equation}
\end{definition}
\noindent The function $\mathcal{S}(t,x)$ represents the probability of no jump occurring in the time interval $[T_n,t)$ given that the process is in the initial state $x$.\\
 Further, we introduce the uniformly distributed random variable $\mathcal{U}$ on $[0,1]$ and define the generalised inverse of $\mathcal{S}(t,x)$ as $\zeta:[0,1]\times E \rightarrow \mathbb{R}_{+}$, given by
\begin{equation}
\label{inverse}
	\zeta(\mathcal{U},x)=\inf\{t\in[0, T] \, ;\, \mathcal{S}(t,x)\leq \mathcal{U}\}.
\end{equation}
Here, $\zeta(\mathcal{U},x)$ serves as a threshold time that determines the occurrence of jumps in the process and their associated probabilities.

\begin{definition}
\label{generalised_inverse_Q}
Let $\psi: [0,1] \times E \rightarrow E$ be a measurable function. We say that $\psi$ is the generalised inverse function of $\mathcal{Q}$ if, for any initial state $x \in E$ and any measurable set $A \in \mathcal{B}(E)$, the following holds
\begin{equation}
\label{inverse_function}
\mathbb{P}(\psi(\mathcal{U},x) \in A) = \mathcal{Q}(x, A).
\end{equation}
\end{definition}
For each $\omega \in \Omega$, the random variable $\psi(\mathcal{U}(\omega),x(\omega))$ describes the post-jump locations of the second component of $x$.\\
For ease of exposition we only discuss the case $d=1$ throughout the rest of the article, the results also hold for $d>1$.

\section{Thinned Euler-Maruyama Method for PDifMP}
\label{section3}

In this section, we build upon the work in \cite{lemaire2020thinning} by adapting the PDMP simulation techniques for PDifMPs. To do that, we propose a simulation technique that we call the thinned Euler-Maruyama method (TEM). This method offers a twofold approach to approximating PDifMPs. Firstly, it addresses scenarios where the flow maps within PDifMPs may not have explicit analytical solutions. Here, we employ the Euler-Maruyama method to approximate the flow maps between jump events. Secondly, we leverage the thinning method to approximate the jump times within the PDifMP. In \cite{davis1984piecewise, davis2018markov}, Davis proposes the use of the generalised inverse, $\zeta$, defined by Equation \eqref{inverse}, to simulate the inter-jump times. The precise computation of $\zeta$ often poses significant challenges because it requires the explicit inversion of the survival function $\mathcal{S}$. While this is straightforward if $\mathcal{S}$ is analytically invertible, such conditions are rarely met in practical scenarios. Consequently, an alternative method is employed to approximate inter-jump times.\\
The thinning method, initially developed for Non-Homogeneous Poisson Processes (NHPPs) \cite{lewis1979simulation}, and later adapted to the simulation of the jump rate of PDMPs, is akin to an acceptance-rejection algorithm, \cite{bierkens2018piecewise,lemaire2018exact, lemaire2020thinning}. It effectively simulates PDifMP jump times by generating potential transition times from a homogeneous Poisson process with rate $\lambda^*$, $\lambda^{*}>0$.
 These times are accepted with a probability determined by the ratio $\lambda(\,\cdot\,)/\lambda^{*}$.\\
 Note that our approach assumes a globally bounded jump rate $\lambda^*$ as an upper bound. However, the use of Poisson thinning can be extended to varying bounds $\lambda^{*}(t)$ along trajectories. This extension is feasible as long as it remains possible to simulate from a NHPP with rate $\lambda^{*}(t)$. We refer to \cite{bouchard2018bouncy, bierkens2019zig, lemaire2018exact} and the references therein for more details about Poisson thinning.

\subsection{Simulation of PDifMPs and thinning}
\label{Thinning_Sec}
Let $T>0$ be a fixed time horizon. Let $(N^{*}_t)_{t\in [0,T]}$ be a homogeneous Poisson process on the probability space $(\Omega, \mathcal{F}, \mathbb{P})$ with intensity $\lambda^{*}$. We assume that $\lambda^{*}$ satisfies the following assumption :
\begin{assumption}
\label{assumptionestimate}
There exists a finite positive constant $\lambda^{*}$ such that, for all $x\in E$, ${\lambda(x)\leq \lambda^{*}}$.
\end{assumption}
\noindent Let $(T^{*}_k)_{k\geq 1}$ denote the successive jump times of $(N^{*}_t)_{t\in [0,T]}$ with $T^{*}_0=0$. Additionally, we define two i.i.d sequences of random variables; $(\mathcal{U}_k)_{k \geq 1}$ and $(\mathcal{V}_k)_{k \geq 1}$, both uniformly distributed on the interval $[0,1]$. These sequences are independent of each other and independent of $(T^{*}_k)_{k\geq 1}$. Note that all these random variables are independent of the Wiener process $(W_t)_{t \in [0,T]}$ used in the model.\\
Let $(x_t)_{t\in [0,T]}$ be a PDifMP satisfying Assumption \ref{Trans_Assump}. Using these sequences, we iteratively construct a sequence of jump times $(T_n)_{n\geq 0}$ and post-jump locations $(y_n,v_n)_{n\geq 0}$ of $(x_t)_{t\in [0,T]}$ as follows.\\
Starting at $T_0$ with a fixed initial state $(y_0,v_0) \in E$, we determine the subsequent jump time $T_1$ of the process $(x_t)_{t\in [0,T]}$ using the thinning technique applied to the sequence $(T^{*}_k)_{k\geq 1}$:
\begin{equation}
\label{first_jump_time}
T_1:= T^{*}_{\tau_1},
\end{equation}
where $\tau_1$ is given by
\begin{equation}
\label{first_waiting_time}
\tau_1:= \inf \left\{ k>0 : \mathcal{U}_k \lambda^{*}\leq \lambda \big(\phi(T^{*}_k,x_0),v_0\big) \right\}.
\end{equation}
Since $\mathbf{V}$ is an at most a countable set, we can write it as $\mathbf{V}=\{k_1, \ldots, k_{\vert \mathbf{V} \vert}\}$, where $\vert\mathbf{V}\vert$ is the cardinal number of the set $\mathbf{V}$. This implies that we can provide a more specific form of the generalised inverse function of $\mathcal{Q}$ given in Definition \ref{generalised_inverse_Q}. To this end, we introduce the following definition.
\begin{definition}
   For $j \in \{1, \ldots, \vert\mathbf{V}\vert\}$, let $a_j$ define the cumulative distribution function of $\mathcal{Q}$ such that, for any $x=(y,v)\in E$, and $\Tilde{y}\in E_1$, satisfying Assumption \ref{Trans_Assump} (b), we have $a_0(x)=0$ and
\begin{equation}\label{eq:4}
a_j(x) = a_j(y,v) := \sum_{i=1}^j \mathcal{Q}\big((y,v),(\Tilde{y} ,k_i)\big).
\end{equation}     
\end{definition}
\noindent Then, for any generic $x\in E$, the generalised inverse $\psi(\,\cdot\,,x)$ of $\mathcal{Q}(x,\,\cdot\,)$, as described in \eqref{inverse_function}, is discretised. Specifically, we define $\psi$ in terms of discrete intervals using the cumulative distribution function $a_j$, which partitions the probability space into segments corresponding to specific outcomes as follows
\begin{equation}\label{eq:5}
    \psi(u,x):= \sum_{i=1}^{|\mathbf{V}|} k_i \mathbbm{1}_{a_{i-1}(x)<u\leq a_i(x)}, \quad \forall \, x \in E, \, \forall \, u \in [0,1].
\end{equation}

\noindent Now, we can construct $(y_1,v_1)$ from the uniform random variable $\mathcal{V}_1$ and the function $\psi$ defined in \eqref{eq:5} as
\begin{align*}
    (y_1,v_1)=& \big(\phi(T^*_{\tau_1},x_0), \psi(\mathcal{V}_1,(\phi(T^*_{\tau_1},x_0),v_0)) \big),\\
    =& \big(\phi(T_{1},x_0), \psi(\mathcal{V}_1,(\phi(T_{1},x_0),v_0)) \big).
\end{align*}

\noindent Conditioning on  $(\tau_1, (T^*_k)_{k\leq \tau_1}, (W_t)_{t \leq T^*_{\tau_1}})$, the distribution of $(y_1, v_1)$ is determined by the transition kernel $\mathcal{Q}$ evaluated at the state $(\phi(T^*_{\tau_1},x_0),v_0)$. In other words, the conditional distribution of $(y_1, v_1)$ is given by
\begin{equation*}
    \sum_{k\in \mathbf{V}} \mathcal{Q}\big((\phi(T^*_{\tau_1},x_0),v_0),(y_1, k)\big).
\end{equation*}

\noindent Let $(\tau_{n-1}, (T^{*}_k)_{k\leq \tau_{n-1}}, (y_{n-1},v_{n-1}))$, $n>1$, be the information obtained up to time $T_{n-1}$. To construct $T_n$, we perform thinning on the point process $(T^{*}_k)$ based on the process being in the state $(y_{n-1},v_{n-1})$ and time $\tau_{n-1}$ for $k\leq \tau_{n-1}$. This is done by removing points $T^{*}_k$ for which the corresponding random variable $\mathcal{U}_k$ is greater than $\lambda\big(\phi(T^{*}_k-T^*_{\tau_{n-1}},x_{n-1}),v_{n-1}\big)/\lambda^{*}$. The resulting thinned point process is denoted by 
\begin{equation}
\label{random_index}
    T_n:= T^*_{\tau_n},
\end{equation}
where 
\begin{equation}
\label{tau_def}
    \tau_n:= \inf \{ k>\tau_{n-1} \,:\, \mathcal{U}_k \lambda^{*}\leq \lambda \big(\phi(T^{*}_k-T^{*}_{\tau_{n-1}},x_{n-1}),v_{n-1}\big) \}.
\end{equation}
\noindent This process repeats for each jump event $n$, and the resulting PDifMP $(y_n,v_n)$ evolves in accordance with the defined jump times and post-jump states such that
\begin{align*}
    (y_{n},v_n)= &\big(\phi(T^*_{\tau_n}-T^*_{\tau_{n-1}},x_{n-1}), \psi(\mathcal{V}_n,(\phi(T^*_{\tau_n}-T^*_{\tau_{n-1}},x_{n-1}),v_{n-1})) \big),\\
    = &\big(\phi(T_{\tau_n}-T_{\tau_{n-1}},x_{n-1}),  \psi(\mathcal{V}_n,(\phi(T_{\tau_n}-T_{\tau_{n-1}},x_{n-1}),v_{n-1})) \big).
\end{align*}

\noindent Next, we define the PDifMP $(x_t)_{t\in[0,T]}$ at time $T_n$ from the sequence $(y_n,v_n)$ by
\begin{equation}\label{margi_x_t}
    x_t:= \big(\phi((t-T_n),x_n),v_n\big), \quad t\in [T_n,T_{n+1}).
\end{equation}
Therefore, for $n\geq 1$, $x_{T_n}=(y_n,v_n)$ and the pre-jump state $x^-_{T_n}=(y_n,v_{n-1})$.\\
Furthermore, we define the counting process $N_t$ associated with the jump times $T_n$ as 
\begin{equation}
\label{point_process}
N_t:=\sum_{n\geq 1} \mathbbm{1}_{T_n\leq t}.
\end{equation}
Note that Assumption \ref{assumptionestimate} ensures the boundedness of the expected number of jumps $N_t$.
\indent In Algorithm \ref{Alg_simulation_PDifMP} (see \ref{App_C}) we have detailed the steps necessary to simulate PDifMPs employing the thinning method coupled with exact solutions for the continuous component updates between jumps.
 \subsection{Approximation of the flow maps}
\label{Euler_Sec}
In practical applications, it is often difficult to obtain explicit expressions for the characteristics of PDifMPs when certain conditions are not met. For an extensive review and detailed discussion on various simulation methods applicable to these characteristics, we refer the reader to \cite{Bertazzietal2022}.\\
In this article, we propose the Euler-Maruyama scheme as our primary method for approximating the function $\phi$ for general applications. For comparative purposes, we also introduce a splitting method applied to a specific model, as discussed in Section \ref{section6}. We combine both methods with the thinning technique for the jump times and analyse this combination with the Euler-Maruyama method (see Sections \ref{section4}-\ref{section5}).\\
We denote by $\overline{\phi}$ the numerical scheme approximating $\phi$. Here, we consider schemes satisfying the following:
\begin{lemma}
\label{phi_lem}
Assume that the assumptions of Theorem \ref{existence} are verified.
Then, for all $(y_1,v) \in E$ and $(y_2,v) \in E$, there exists a positive constant $C_2$ independent of $v$ and $h$ such that
\begin{equation}\label{estimator_phi-phi}
    \mathbb{E}  [\sup_{t\in [0,T]}|\phi(t,(y_1,v))-\phi(t,(y_2,v))|^2] \leq e^{C_1T} |y_1-y_2|^2,
\end{equation}
where $y_1$ and $y_2$ refer to the initial conditions, and they are updated after each jump.
\end{lemma}
\begin{proof}
We refer to \ref{App_A} for the proof of Lemma \ref{phi_lem}.
\end{proof}
\begin{theorem}{(Yuan et al. \cite{yuan2004convergence})}
\label{eul_est}
Assuming the global Lipschitz conditions \eqref{lipschitz_cond} in Theorem \ref{existence}, we have
\begin{equation}
    \mathbb{E}[\sup_{t\in [0,T]}|\phi(t,x)-\overline{\phi}(t,x)|^2] \leq C_2h , \quad  \forall x\in E,
\end{equation}
where $C_2$ is a positive constant independent of $v$ and $h$.
\end{theorem}
\noindent Following the idea in \cite{lemaire2020thinning}, we introduce the following lemmas, to provide a quantitative estimate for the difference between the exact solution $\phi$ and the numerical solution $\overline{\phi}$ in terms of the initial conditions and the time step size.
\begin{lemma}\label{lemma-combined}
Assume that the conditions of Theorem \ref{existence} hold. Let $\phi(t,x) = \phi(t,(y_1,v))$ and $\overline{\phi}(t,x) = \overline{\phi}(t,(y_2,v))$ be the exact and approximating solutions, respectively, of the system \eqref{Sde_e} with deterministic initial conditions $(y_1,y_2) \in E_1^2$.
\begin{enumerate}
\item For all fixed $\omega \in \Omega$, there exist two positive constants $C_1$ and $C_2$ independent of $v$ and $h$ such that the following holds
\begin{equation}\label{estimator_phi}
\sup_{t\in [0,T]}|\phi(t,x)-\overline{\phi}(t,x)|\leq e^{C_1T} |y_1-y_2| + C_2h , \quad \forall v \in \mathbf{V}, \, \, \forall \, (y_1,y_2)\in E_1^2.
\end{equation}
\item For all fixed $t \in [0,T]$, there exist two positive constants $C_3$ and $C_4$ independent of $v$ and $h$ such that the following holds
\begin{equation}\label{estimator_phi_random}
\mathbb{E} [\sup_{t\in [0,T]}|\phi(t,x)-\overline{\phi}(t,x)|^2] \leq e^{C_3T} |y_1-y_2|^2 + C_4h , \quad \forall \, v\in \mathbf{V}, \, \, \forall \, (y_1,y_2)\in E_1^2.
\end{equation}
\end{enumerate}
\end{lemma}
\begin{proof}
The proof is straightforward using Theorem \ref{eul_est} and Lemma \ref{phi_lem}.
\end{proof}

\noindent To the family $\phi$ we associate a PDifMP constructed as in Section \ref{Thinning_Sec} that we denote $\overline{x}_t$. Throughout our analysis, we assume that $b$ and $\sigma$ satisfy the assumptions stated in Theorem \ref{existence} for all $v \in \mathbf{V}$. For any $x=(y,v)\in E$, we denote by $\phi(t,x)$, $t\in [0,T]$, the unique solution of the SDE

\begin{equation}
     \label{Sde_e}
        \left\{
    \begin{array}{ll}
d\phi(t,x_t)&=b(\phi(t,x_t),v_t)dt+\sigma(\phi(t,x_t),v_t)dW_t,\\
\phi(0,x)&=y,
 \end{array}
\right. 
    \end{equation}
where $y$ can be a constant or an $\mathcal{F}_0$-measurable random variable. Note that, the $0$ in $\phi(0, x)$ is a symbolic reset point, which updates $y$ to the system state at the beginning of each jump, thereby dynamically setting the initial condition for the evolution of $\phi$.\\
\noindent Now, given a time step $h>0$, we define the partition points of the interval $[0,T]$ as $\overline{t}_i=ih$ for $i\geq 0$. For all $x=(y,v) \in E$, we iteratively construct the discrete-time process $\overline{y}_i$ as follows:

\begin{equation}
\overline{y}_{i+1} = \overline{y}_{i} + b(\overline{y}_{i},v)h+ \sigma(\overline{y}_{i},v)\Delta W_i,
\end{equation}
where $\Delta W_i=W_{\overline{t}_{i+1}}-W_{\overline{t}_{i}}$ is the increment of a standard Wiener process between times $\overline{t}_{i}$ and $\overline{t}_{i+1}$.\\
For any $x=(y,v)\in E$ and $t \in [\overline{t}_i,  \overline{t}_{i+1}]$, we set the continuous time interpolation of the numerical solution

\begin{equation}
\label{varphi_bar}
\overline{\phi}(t,x)=\overline{y}_{i} + b(\overline{x}_{i})(t-\overline{t}_i)+ \sigma(\overline{x}_{i})( W_t-W_{\overline{t}_i}).
\end{equation}
In other words, the scheme approximates the coefficients $b$ and $\sigma$ as piecewise constant within each time interval $t \in [\overline{t}_i,  \overline{t}_{i+1}]$. Note that $\overline{y}_i$ in Equation \eqref{varphi_bar} represents the starting value for each successive time interval $[t_i, t_{i+1}]$, rather than the post-jump values.\\ 
To ensure that the flow maps and their approximations remain bounded in expectation for all $v \in \mathbf{V}$, we introduce the following lemma.

\begin{lemma}{(Mao \cite{mao2007stochastic})}
\label{est_1}
Assuming the linear growth condition (\ref{linear_growth}) holds for all $x=(y,v)\in E$ with $y=\overline{y}$, there exists a positive constant $K$ such that for fixed $\omega \in \Omega$, the functions $\phi(t,x)$ and $\overline{\phi}$ satisfy 
\begin{equation*}
    \mathbb{E}[\sup_{t\in[0,T]} \lvert \phi(t,x) \rvert ^2] \leq (1+3\mathbb{E}\lvert y \rvert^2) e^{3KT(T+1)},
\end{equation*}
and 
\begin{equation*}
    \mathbb{E}[\sup_{t\in[0,T]} \lvert \overline{\phi}(t,x) \rvert ^2] \leq (1+3\mathbb{E}\lvert y \rvert^2) e^{3KT(T+1)}.
\end{equation*}
\end{lemma}
\subsection{Thinned Euler Maruyama Algorithm}
\noindent The purpose of this construction is to approximate the solution of the PDifMP $(x_t)_{t\in[0, T]}$ by a process $(\overline{x}_t)_{t\in[0, T]}$. The approach involves simulating the continuous part of $(x_t)_{t\in [0,T]}$ using the Euler-Maruyama scheme with a step size $h$ and determining the jump times by simulating a Poisson process with intensity function $\lambda(\overline{\phi}(t,x))$. The simulation of the piecewise diffusion Markov process \((x_t)_{t \in [0, T]}\) is constructed as follows:
\begin{enumerate}
    \item \textbf{Initialisation and discretisation:}
    At time \(t=0\), initialise the continuous stochastic component \(y_t\) of \((x_t)_{t \in [0, T]}\) at \(y_0 = y\). Employ the Euler-Maruyama scheme to discretise the continuous trajectory and define the piecewise linear interpolation \(\overline{\phi}(t, x)\) as outlined in Equation (\ref{varphi_bar}). This sets up the framework for approximating the path of the continuous component between jumps.
    \item \textbf{Continuous evolution until the first jump:}
    The continuous component \(y_t\) follows the trajectory determined by \(\overline{\phi}(t, x_0)\) up to the occurrence of the first jump at time \(\overline{T}_1\). This jump time is calculated using a non-homogeneous Poisson process. The waiting time \(\tau_1\) for this jump is derived from an exponential distribution, reflecting the current state-dependent rate, as specified in Equations (\ref{first_jump_time}) and (\ref{first_waiting_time}). Here, we replace \(\phi(T^{*}_k, x_0)\) by \(\overline{\phi}(T^{*}_k, x_0)\) to accommodate the discretised process.
    \item \textbf{State update at the first jump time:} At time $\overline{T}_1$, after iteratively applying the Euler-Maruyama scheme with step size $h$ to simulate the continuous trajectory from $t=0$, the continuous component of $\overline{x}_{\overline{T}_1}$ is updated to the value calculated by the final application of the scheme, i.e. $y_1=\overline{\phi}(\overline{T}_1, x_0)$. This update effectively captures the state of the continuous component at the jump time, providing a new starting point for subsequent evolution.
    \item \textbf{Post-jump evolution:}
    After the jump, update the discrete jump component state from \(\overline{v}_0\) to \(\overline{v}_1\), marking the transition caused by the jump. Resume the simulation using the new flow \(\overline{\phi}(t - \overline{T}_1, \overline{x}_1)\) from the updated state until the next jump time \(\overline{T}_2\) occurs. This iterative process continues, computing the flow after each jump based on the updated starting state at each jump time, until the end of the simulation period at time \(T\).
\end{enumerate}

\noindent In this context, the Euler-Maruyama scheme is applied iteratively in discrete intervals of size $h$, where $h \leq \overline{T}_n$, $n\geq 1$, to simulate the continuous dynamics of the process up to the first jump time. The number of iterations depends on the chosen $h$ and the duration of the interval $[\overline{T}_n, \overline{T}_{n+1}]$. Furthermore, due to the dependence of the intensity function on the entire path of the continuous state, the Euler-Maruyama scheme has to be applied at each iteration step.

\begin{remark}
 The thinning method used in our simulation allows for the detection and processing of multiple jumps within a single time step, should they occur. Each jump time, generated independently by the Poisson process, is evaluated using the thinning acceptance criteria (\(\mathcal{U}_k \lambda^{*} \leq \lambda(\overline{\phi}(T^{*}_k, x_0), v_0)\)). If multiple jumps pass this test during the same interval, they are processed sequentially to ensure that the state accurately reflects the impact of each jump before proceeding.
\end{remark}

\begin{notation}
Hereafter, we denote by $\overline{T}_n$ the sequence of jump times of the process $(\overline{x}_t)_{t\in [0,T]}$ and by $(\overline{N}_t)_{t\in [0,T]}$ the associated counting process.
\end{notation}
\begin{remark}{(Adaptive Grid)}
Our approach, in contrast to traditional Euler-Maruyama methods with fixed time grids, employs a jump-adapted time discretisation scheme as in \cite{platen2010jump}. This method dynamically adjusts the discretisation grid to match the process dynamics. The time points of interest $ \overline{T}_n, \{n=0,\ldots,\overline{N}_T\} $ are determined by \eqref{random_index} and \eqref{tau_def}, where $ \overline{N}_T $ is the total number of intervals in the adaptive grid. For each interval $ [\overline{T}_n, \overline{T}_{n+1}) $, we introduce grid points 
\begin{equation}
    \label{grid_points}
    \{ \overline{T}_{n}+kh\}, \quad k=0,\ldots,\left(\frac{(\overline{T}_{n+1}\wedge T)-\overline{T}_n}{h}\right),
\end{equation}

to ensure an adaptive number of subdivisions based on the process variability and the rate of change.
\end{remark}
\subsection{Coupling and Convergence to the True Process}
While the discretised approximation provides a practical approach to simulating PDifMPs, it is also essential to consider how closely these approximations adhere to the true underlying process. Here, we explore the coupling between the true process and its approximation primarily to analyse the convergence properties of our simulation method.\\
Both the true process and its discretised approximation are derived from the same underlying stochastic drivers: \(N^{*}_t\), \(\mathcal{U}_k\), and \(\mathcal{V}_k\). Despite this common foundation, there is a probability that they may jump at different times or transition into different states due to the discretisation errors and the approximation methods used.\\
However, in specific scenarios where the jump characteristics \((\lambda, \mathcal{Q})\) of the PDifMP are solely dependent on the discrete component \(v\), we ensure that both the true process and its approximation exhibit identical jump timings and state transitions for the \(v\)-component. This alignment is crucial for models where the jump dynamics are critical to the behaviour and analysis of the system.\\
These concepts are further elaborated in \cite{durmus2021piecewise}, primarily focusing on PDMPs, but their extension to PDifMPs is straightforward.
\subsection{Preliminary Tools and Results}
\label{preliminary_tools}
Here, we introduce key definitions, theorems, and lemmas that provide the groundwork for subsequent detailed analyses of convergence in the mean square sense and weak error expansion for PDifMP simulations.
\begin{definition}
\label{taudef}
Analogously to \cite{lemaire2020thinning}, we define $\overline{\tau}^{\star}:=\inf \{ k>0 : (\tau_k,v_k) \neq (\overline{\tau}_k,\overline{v}_k) \}$.
\end{definition}
\noindent The random variable $\overline{\tau}^{\star}$ allows us to partition the trajectories of the couple of processes $(x_t, \overline{x}_t)$ based on whether the first thinned point in the approximating process occurs at the same time as in the true process. More specifically, we can define the following event:
\begin{equation}
\label{T_events}
    \{ \min(T_{\overline{\tau}^{\star}},\overline{T}_{\overline{\tau}^{\star}})>T \}=\{ N_T=\overline{N}_T, \, (T_1,v_1)=(\overline{T}_{1},\overline{v}_{1}),\ldots, (T_{N_T},v_{N_T})=(\overline{T}_{\overline{N}_T},\overline{v}_{\overline{N}_T})\}.
\end{equation}
\begin{lemma}
\label{betalem_1}
Let $\phi$ and $\overline{\phi}$ satisfy the conditions of Lemma \ref{lemma-combined}. Let $(t_n)_{n\geq 0}$ be an increasing sequence of non-negative real numbers with $t_0=0$ and let $(\alpha_n)_{n\geq 0}$ be a sequence of $\mathbf{V}$-valued components. For a given $y\in E_1$, we define iteratively the sequences $(\beta_n)_{n\geq 0}$ and $(\overline{\beta}_n)_{n\geq 0}$ as follows
 \begin{equation*}
        \left\{
    \begin{array}{ll}
 \beta_n &= \phi((t_n-t_{n-1}), (\beta_{n-1}, \alpha_{n-1})),\\
 \beta_0&=y,
 \end{array}
\right. 
    \end{equation*}
    and 
    \begin{equation*}
        \left\{
    \begin{array}{ll}
 \overline{\beta}_n &= \overline{\phi}((t_n-t_{n-1}), (\overline{\beta}_{n-1}, \alpha_{n-1})),\\
 \overline{\beta}_0&=y.
 \end{array}
\right. 
    \end{equation*}
\begin{enumerate}
    \item For all fixed $\omega \in \Omega$, for all $n\geq 1$ and for all $C_5>0$ and $C_6>0$, we have
\begin{equation}
    \label{Beta_est_fix_w}
     | \overline{\beta}_n-\beta_n| \leq e^{C_5t_n} n C_6 h , 
\end{equation}
   where $C_5$ and $C_6$ are independent of the discretisation step size $h$.
   \item For all fixed $t\in [0,T]$, for all $n\geq 1$ and for all $C_7>0$ and $C_8>0$, we have
\begin{equation}
    \label{Beta_est_fix_t}
\mathbb{E}[ | \overline{\beta}_n-\beta_n|]^2 \leq e^{C_7t_n} n C_8 h , 
\end{equation}
    where $C_7$ and $C_8$ are independent of the discretisation step size $h$.
\end{enumerate}
\end{lemma}
\begin{proof}
We present here the proof of the second part of Lemma \ref{betalem_1}, where $\eqref{Beta_est_fix_t}$ is to be shown. The proof for the first part is similar and can be found in \cite{lemaire2020thinning}.\\
Let $n\geq 1$, $C_8>0$, and $C_9>0$. From estimate $\eqref{estimator_phi_random}$, we have for all $k\leq n$,
    \begin{equation*}
    \vert \overline{\beta}_k- {\beta}_k  \vert ^2 \leq e^{C_1(t_k-t_{k-1})} \vert \overline{\beta}_{k-1}- {\beta}_{k-1}  \vert ^2 +C_2 h,
\end{equation*}
by rearranging the terms, we obtain
\begin{equation}
\label{eq_n}
   e^{-C_1t_{k}} \vert \overline{\beta}_k- {\beta}_k  \vert ^2 \leq e^{-C_1t_{k-1}} \vert \overline{\beta}_{k-1}- {\beta}_{k-1}  \vert ^2 +C_2 h.
\end{equation}
Since $\overline{\beta}_0={\beta}_0$, we sum up $\eqref{eq_n}$ for $1\leq k \leq n$ to obtain
\begin{equation*}
    \mathbb{E} [ | \overline{\beta}_n-\beta_n|^2] \leq e^{C_1t_n}n C_2 h.
    \end{equation*}
\end{proof}

\section{Convergence in the mean-square sense}
\label{section4}
\label{Scheme_Convergence}
The main result of this section is Theorem \ref{Strong_thm} that establishes mean-square error estimates for the approximation scheme. Let us first present the essential assumptions on the process characteristics required to show Theorem \ref{Strong_thm}.
\begin{assumption}
\label{MSE_Ass1}
Assume that $b$ and $\sigma$ satisfy condition (\ref{lipschitz_cond}) of Theorem \ref{existence} and let $\lambda$ verify Assumption \ref{assump_lambda}.
For all $v \in \mathbf{V}$ and all $B \in \mathcal{B}(\mathbf{V})$, the following hold
\begin{enumerate}
    \item for any $y_1, y_2 \in E_1\times E_1$ there exists a positive constant $C_{\lambda}$ independent of $v$, such that
    \begin{equation}
        \vert \lambda(y_1,v)- \lambda(y_2,v)   \vert \leq C_{\lambda} \vert y_1-y_2 \vert,
    \end{equation}
    \item  for any $y_1, y_2 \in E_1\times E_1$ there exists a positive constant $C_{\mathcal{Q}}$ independent of $v$, such that
     \begin{equation}
     \label{Lip_Q}
           \vert \mathcal{Q}(y_1,v)- \mathcal{Q}(y_2,v)   \vert \leq C_{\mathcal{Q}} \vert y_1-y_2 \vert.
     \end{equation}
\end{enumerate}
\end{assumption}

\begin{assumption}
\label{Lipschitz_H}
Let $F: E \rightarrow \mathbb{R}$ be a bounded measurable function. For any $x=(y,v) \in E$ the function $F$ is globally Lipschitz.
\end{assumption}

\begin{theorem}
\label{Strong_thm}
Let the flow maps $\phi$ and $\overline{\phi}$ satisfy Lemma \ref{lemma-combined} and let $(x_t)_{t\in [0,T]}$ be a PDifMP and $(\overline{x}_t)_{t\in [0,T]}$ be its approximation constructed as described in Section \ref{section3} with $x_0=\overline{x}_0=x$ for some $x\in E$. Let $F$ be a bounded measurable function on $E$ satisfying 
Assumptions \ref{Trans_Assump}, \ref{assumptionestimate} and \ref{MSE_Ass1}. Then, there exists a positive constant $L_1$ independent of $h$ and $v$ such that
\begin{equation}
\label{strong_est}
\mathbb{E}\big[ |F(\overline{x}_T)-F(x_T)|^2 \big]\leq L_1 h.
\end{equation}
\end{theorem}
\begin{proof}
The proof consists of two steps, with the main part being inspired by \cite{lemaire2020thinning}. Using Definition \ref{taudef}, the term in the left-hand side of (\ref{strong_est}) can be written as follows
\small
\begin{align*}
  \mathbb{E}\big[ |F(\overline{x}_T)-F(x_T)|^2 \big]&=\mathbb{E}\big[ |F(\overline{x}_T)-F(x_T)|^2\mathbbm{1}_{\min(T_{\overline{\tau}^{\star}},\overline{T}_{\overline{\tau}^{\star}})\leq T} \big] +\mathbb{E}\big[ |F(\overline{x}_T)-F(x_T)|^2 \mathbbm{1}_{\min(T_{\overline{\tau}^{\star}},\overline{T}_{\overline{\tau}^{\star}})>T} \big]\\
  &=:D_1+D_2.
\end{align*}
\normalsize
\noindent The order of $D_1$ is given by the order of the probability that the discrete processes $(T_n,v_n)$ and $(\overline{T}_n,\overline{v}_n)$ differ on the interval $[0,T]$. The order of $D_2$ is determined by the order of the Euler-Maruyama scheme squared under the assumption that $(T_n,v_n)$ and $(\overline{T}_n,\overline{v}_n)$ remain equal on $[0,T]$. Thus, we proceed to show that both $D_1$ and $D_2$ are of order $h$.\\

\noindent \textit{Step 1: Estimation of $D_1$}.\\

\noindent Under Assumption \ref{Lipschitz_H}, the function $F$ is bounded with a constant $M_F>0$. Consequently, by applying the triangle inequality, we obtain
$${\lvert F(\overline{x}_T)-F(x_T) \rvert^2 \leq   4 M^2_F}.$$
This implies that
\begin{equation*}
 D_1\leq 4M^2_F\mathbb{P}\big(\min(T_{\overline{\tau}^{\star}},\overline{T}_{\overline{\tau}^{\star}}) \leq T \big).
\end{equation*}
\noindent Furthermore, for $k\geq 1$, we have
\begin{equation*}
    \{\overline{\tau}^{\star}=k \}=\{ \overline{\tau}^{\star}>k-1 \} \cap \{(\tau_k,v_k)\neq (\overline{\tau}_k,\overline{v}_k)\}.
\end{equation*}
\noindent Hence,
\begin{align*}
 \mathbb{P}\left(\min \left(T_{\overline{\tau}^{\star}}, \overline{T}_{\overline{\tau}^{\star}}\right) \leq T\right)
 &=\sum_{k \geq 1} \mathbb{E}\left[\mathbbm{1}_{\min \left(T_{k}, \overline{T}_{k}\right) \leq T} \mathbbm{1}_{\overline{\tau}^{\star}=k}\right],\\
 &=\sum_{k \geq 1} \mathbb{E}\left[\mathbbm{1}_{\min \left(T_{k}, \overline{T}_{k}\right) \leq T} \mathbbm{1}_{\overline{\tau}^{\star}>k-1} \mathbbm{1}_{\left(\tau_{k}, v_{k}\right) \neq \left(\overline{\tau}_{k}, \overline{v}_{k}\right)}\right],\\
 &=\sum_{k \geq 1} \mathbb{E}\left[\mathbbm{1}_{\min \left(T_{k}, \overline{T}_{k}\right) \leq T} \mathbbm{1}_{\overline{\tau}^{\star}>k-1} \left(\mathbbm{1}_{\tau_{k} \neq \overline{\tau}_{k}} + \mathbbm{1}_{\tau_{k} = \overline{\tau}_{k}} \mathbbm{1}_{v_{k} \neq \overline{v}_{k}}\right)\right],\\
 &=\sum_{k \geq 1} \mathbb{E}\left[\mathbbm{1}_{\min \left(T_{k}, \overline{T}_{k}\right) \leq T} \mathbbm{1}_{\overline{\tau}^{\star}>k-1} \left(\mathbbm{1}_{\tau_{k} \neq \overline{\tau}_{k}}\right)\right] \\
 & \quad+ \sum_{k \geq 1} \mathbb{E}\left[\mathbbm{1}_{\min \left(T_{k}, \overline{T}_{k}\right) \leq T} \mathbbm{1}_{\overline{\tau}^{\star}>k-1} \left(\mathbbm{1}_{\tau_{k} = \overline{\tau}_{k}} \mathbbm{1}_{v_{k} \neq \overline{v}_{k}}\right)\right],\\
 &=\sum_{k \geq 1} \left(\mathbb{E}\left[\mathbbm{1}_{\min \left(T_{k}, \overline{T}_{k}\right) \leq T} \mathbbm{1}_{\overline{\tau}^{\star}>k-1} \left(\mathbbm{1}_{\tau_{k} \neq \overline{\tau}_{k}} (\mathbbm{1}_{v_{k} \neq \overline{v}_{k}} + \mathbbm{1}_{v_{k} = \overline{v}_{k}}) \right)\right]\right)\\
& \quad +\sum_{k \geq 1} \left( \mathbb{E}\left[\mathbbm{1}_{\min \left(T_{k}, \overline{T}_{k}\right) \leq T} \mathbbm{1}_{\overline{\tau}^{\star}>k-1} \mathbbm{1}_{\tau_{k} = \overline{\tau}_{k}} \mathbbm{1}_{v_{k} \neq \overline{v}_{k}}\right]\right),\\
& \leq \sum_{k \geq 1} \overline{J}_{k}+2 \overline{I}_{k} ,
\end{align*}
where

\begin{equation}
\small{
\overline{J}_{k}:=\mathbb{E}\left[\mathbbm{1}_{\min \left(T_{k}, \overline{T}_{k}\right) \leq T} \mathbbm{1}_{\overline{\tau}^{\star}>k-1} \mathbbm{1}_{\tau_{k}=\overline{\tau}_{k}} \mathbbm{1}_{v_{k} \neq \overline{v}_{k}}\right]  \quad \text{and} \quad \overline{I}_{k}:=\mathbb{E}\left[\mathbbm{1}_{\min \left(T_{k}, \overline{T}_{k}\right) \leq T} \mathbbm{1}_{\overline{\tau}^{\star}>k-1} \mathbbm{1}_{\tau_{k} \neq \overline{\tau}_{k}}\right].}
\label{J_et_I}
\end{equation}
\newline
\noindent We consider the first term $\overline{J}_{k}$. For $k \geq 1$, we have $\left\{\tau_{k}=\overline{\tau}_{k}\right\}=\left\{T_{k}=\overline{T}_{k}\right\}$, and on the event $\left\{T_{k}=\overline{T}_{k}\right\}$, we have $\min \left(T_{k}, \overline{T}_{k}\right)=T_{k}=\overline{T}_{k}$. Thus, $\overline{J}_{k}$ can be written as
\begin{equation*}
    \overline{J}_{k}=\mathbb{E}\left[\mathbbm{1}_{T_{k} \leq T} \mathbbm{1}_{\overline{\tau}^{\star}>k-1} \mathbbm{1}_{\tau_{k}=\overline{\tau}_{k}} \mathbbm{1}_{v_{k} \neq \overline{v}_{k}}\right].
\end{equation*}

\noindent Since $\overline{\tau}^{\star}>k-1$ if and only if $\left(\tau_{i}, v_{i}\right)=\left(\overline{\tau}_{i}, \overline{v}_{i}\right)$ for all $i=0,\ldots, k-1$, we can express $\overline{J}_{k}$ as

\begin{equation}
\label{J_K}
 \overline{J}_{k}=\sum_{\substack{1 \leq p_{1}<\ldots<p_{k}\\  \alpha_{1}, \ldots, \alpha_{k-1} \in \mathbf{V}}} \mathbb{E}\left[\underbrace{\mathbbm{1}_{\left\{\tau_{i}=\overline{\tau}_{i}=p_{i}, 1 \leq i \leq k\right\}} \mathbbm{1}_{\left\{v_{i}=\overline{v}_{i}=\alpha_{i}, 1 \leq i \leq k-1\right\}} \mathbbm{1}_{(T^{*}_{p_{k}}\leq T)}}_{:=P^{J}_k} \mathbbm{1}_{(v_{k} \neq \overline{v}_{k})}\right] .   
\end{equation}
\noindent Recall the sequence of jump times $(T^{*}_k)_{k\geq 1}$ and the two i.i.d sequences of random variables $(\mathcal{U}_k)_{k \geq 1}$ and $(\mathcal{V}_k)_{k \geq 1}$ defined in Section \ref{Thinning_Sec}. The discrete processes $v_k$ and $\overline{v_k}$ are given by
\begin{equation*}
v_k=(\psi(y_k,v_{k-1}), \mathcal{V}_k), \qquad \overline{v}_k=(\psi(\overline{y}_k,\overline{v}_{k-1}), \mathcal{V}_k).
\end{equation*}

\noindent  Let $\mathcal{J}=(\mathcal{U}_i, T_{j}^{*}, \mathcal{V}_q)$ for $1 \leq i \leq p_{k}$, $1 \leq j \leq p_{k}$, and $1 \leq q \leq k-1$. The vector $\mathcal{J}$ is independent of $\mathcal{V}_{k}$. Since the random variable $P^{J}_k$ depends on $\mathcal{J}$, conditioning by $\mathcal{J}$ in \eqref{J_K} and applying Lemma \ref{useful_lem} (see \ref{App_B}), we obtain

\begin{equation}
  \begin{aligned}
& \mathbb{E}\left[ P^{J}_k \mathbbm{1}_{v_{k} \neq \overline{v}_{k}}\right]
& \leq \mathbb{E}\left[P^{J}_k \sum_{j=1}^{|\mathbf{V}|-1}\lvert a_{j}\left(\overline{y}_{k},\alpha_{k-1}\right)-a_{j}\left(y_{k},\alpha_{k-1}\right)\rvert\right] . 
\end{aligned}
\label{B2_inq}
\end{equation}
\noindent Using the definition of $a_j$ (see Equation \eqref{eq:4}), the triangle inequality and Assumption \ref{MSE_Ass1}, we have
\begin{equation*}
    \sum_{j=1}^{|\mathbf{V}|-1}\left|a_{j}\left(\overline{y}_{k},\alpha_{k-1}\right)-a_{j}\left(y_{k},\alpha_{k-1}\right)\right| \leq \FTS{(|\mathbf{V}|-1)|\mathbf{V}|}{2} C_{\mathcal{Q}}\left|\overline{y}_{k}-y_{k}\right|,
\end{equation*}
where $C_{\mathcal{Q}}$ is the Lipschitz constant for $\mathcal{Q}$. Given that in \eqref{B2_inq} we are on the event
\begin{equation*}
   \left\{\tau_{i}=\overline{\tau}_{i}=p_{i}, 1 \leq i \leq k\right\} \bigcap\left\{v_{i}=\overline{v}_{i}=\alpha_{i}, 1 \leq i \leq k-1\right\}, 
\end{equation*}
applying Lemma \ref{betalem_1} implies 
\begin{equation*}
   \left|\overline{y}_{k}-y_{k}\right| \leq e^{L T_{p_{k}}^{*}}k C h. 
\end{equation*}
\noindent Therefore, $\overline{J}_{k} \leq C_{1} h \mathbb{E}\left[k\mathbbm{1}_{(T_{k} \leq T)} \right]$, where $C_{1}$ is a constant independent of $h$. The summation $\sum_{k \geq 1} k\mathbbm{1}_{(T_{k} \leq T)}$ represents the total number of jump times up to time $T$. Since the number of jump times $N_T$ is finite (see Assumption \ref{assumptionestimate}), we have
\begin{equation*}
  \sum_{k \geq 1} k\mathbbm{1}_{(T_{k} \leq T)}=\sum_{k=1}^{N_{T}} k \leq N_{T}^{2} \qquad \text{and} \qquad  \mathbb{E}\left[N_{T}^{2}\right] \leq \mathbb{E}\left[\left(N_{T}^{*}\right)^{2}\right]< +\infty.
\end{equation*}
Hence,
\begin{equation}
    \sum_{k \geq 1} \overline{J}_{k}=O\left(h\right).
    \label{Est_J}
\end{equation}
\noindent We now focus on the second term $\overline{I}_{k}$ in Equation \eqref{J_et_I}. Since  $\tau_k \neq \overline{\tau}_k$, this implies either $\tau_k>\overline{\tau}_k$, or $\tau_k<\overline{\tau}_k$, thus we can express $\overline{I}_{k}$ as

\begin{equation*}
   \begin{aligned}
\overline{I}_{k} &=\mathbb{E}\left[\mathbbm{1}_{\min \left(T_{k}, \overline{T}_{k}\right) \leq T} \mathbbm{1}_{\overline{\tau}^{\star}>k-1}\mathbbm{1}_{\tau_{k}<\overline{\tau}_{k}}+\mathbbm{1}_{\min \left(T_{k}, \overline{T}_{k}\right) \leq T} \mathbbm{1}_{\overline{\tau}^{\star}>k-1}\mathbbm{1}_{\tau_{k}>\overline{\tau}_{k}}\right].
\end{aligned} 
\end{equation*}

\noindent Given that  $\left\{\tau_{k}<\overline{\tau}_{k}\right\}=\left\{T_{k}<\overline{T}_{k}\right\}$ and $\left\{\tau_{k}>\overline{\tau}_{k}\right\}=\left\{T_{k}>\right.$ $\left.\overline{T}_{k}\right\}$, we can further simplify $\overline{I}_{k}$ as
\begin{equation*}
  \begin{aligned}
\overline{I}_{k}&=\mathbb{E}\left[\mathbbm{1}_{T_{k} \leq T} \mathbbm{1}_{\overline{\tau}^{\star}>k-1} \mathbbm{1}_{\tau_{k}<\overline{\tau}_{k}}\right]+\mathbb{E}\left[\mathbbm{1}_{\overline{T}_{k} \leq T} \mathbbm{1}_{\overline{\tau}^{\star}>k-1} \mathbbm{1}_{\tau_{k}>\overline{\tau}_{k}}\right], \\
&=: \overline{I}_{k}^{(1)}+\overline{I}_{k}^{(2)}.
\end{aligned}  
\end{equation*}
\noindent To simplify the presentation, we will focus on one of the two terms in $\overline{I}_{k}$, namely $\overline{I}_{k}^{(1)}$. The analysis of $\overline{I}_{k}^{(2)}$ can be done analogously by exchanging the roles of $\left(\tau_{k}, T_{k}\right)$ and $\left(\overline{\tau}_{k}, \overline{T}_{k}\right)$.\\
Similarly to the previous case, we can rewrite $\overline{I}_{k}^{(1)}$ as
\small
\begin{equation}
\label{I_1}
    \overline{I}_{k}^{(1)}=\sum_{\substack{1 \leq p_{1}<\ldots<p_{k}\\ \alpha_{1}, \ldots, \alpha_{k-1} \in \mathbf{V}}} \mathbb{E}\left[\underbrace{\mathbbm{1}_{\left\{\tau_{i}=\overline{\tau}_{i}=p_{i}, 1 \leq i \leq k-1\right\}} \mathbbm{1}_{\left\{v_{i}=\overline{v}_{i}=\alpha_{i}, 1 \leq i \leq k-1\right\}} \mathbbm{1}_{(T^{*}_{p_{k}}\leq T)}}_{:=P^{I}_k} \mathbbm{1}_{\tau_{k}=p_{k}} \mathbbm{1}_{p_{k}<\overline{\tau}_{k}}\right].
\end{equation}
\noindent Further, we have
\small
\begin{multline*}
    \{\tau_{k}=p_{k}\} \cap\{p_{k}<\overline{\tau}_{k}\} \subseteq \\ \left\{\lambda\bigg( \overline{\phi}\big(T_{p_{k}}^{*}-T_{p_{k-1}}^{*}, (\overline{y}_{k-1},\alpha_{k-1})\bigg), \alpha_{k-1}\big)< U_{p_{k}} \lambda^{*} \leq \lambda\bigg( \phi\big(T_{p_{k}}^{*}-T_{p_{k-1}}^{*}, (y_{k-1},\alpha_{k-1})\big),\alpha_{k-1}\bigg)\right\}.
\end{multline*}
\noindent Let $\mathcal{I}=(\mathcal{U}_i, T_{j}^{*}, \mathcal{V}_q)$ for $1 \leq i \leq p_{k-1}$, $1 \leq j \leq p_{k}$, $1 \leq q \leq k-1$. The vector $\mathcal{I}$ is independent of $\mathcal{U}_{p_k}$. Since the random variable $P^{I}_k$ depends on $\mathcal{I}$, conditioning by $\mathcal{I}$ in \eqref{I_1}, we obtain

\[
\begin{aligned}
&\mathbb{E}\left[P^{I}_k \mathbbm{1}_{\tau_{k}=p_{k}} \mathbbm{1}_{p_{k}<\overline{\tau}_{k}}\right] \leq \\
&\qquad \mathbb{E}\left[P^{I}_k \times \left|\lambda\left( \overline{\phi}\left(T_{p_{k}}^{*}-T_{p_{k-1}}^{*}, (\overline{y}_{k-1},\alpha_{k-1})\right),\alpha_{k-1}\right)-\lambda\left( \phi\left(T_{p_{k}}^{*}-T_{p_{k-1}}^{*},( y_{k-1},\alpha_{k-1})\right),\alpha_{k-1}\right)\right|\right].
\end{aligned}
\]

\noindent Applying the Lipschitz continuity of $\lambda$ and Lemma \ref{betalem_1}, we obtain 
\begin{equation*}
    \overline{I}_{k}^{(1)} \leq C_{2} h \mathbb{E}\left[\mathbbm{1}_{T_{k} \leq T} k\right],
\end{equation*}
 where $C_{2}$ is a constant independent of $h$. A similar bound can be obtained for $\overline{I}_{k}^{(2)}$. Therefore, we have
 \begin{equation*}
    \label{bar_I}
    \sum_{k \geq 1} \overline{I}_{k}=O\left(h\right).
\end{equation*}
\noindent \textit{Step 2:Estimation of $D_2$.}\\

\noindent For $n \geq 0$ we have
\begin{equation*}
    \left\{N_{T}=n\right\} \cap\left\{\min \left(T_{\overline{\tau}^{\star}}, \overline{T}_{\overline{\tau}^{\star}}\right)>T\right\}=\left\{N_{T}=n\right\} \cap\left\{\overline{N}_{T}=\right.n\} \cap\left\{\overline{\tau}^{\star}>n\right\},
\end{equation*}
 where we can interchange $\left\{N_{T}=n\right\}$ and $\left\{\overline{N}_{T}=n\right\}$. Hence, using the partition $\left\{N_{T}=n, n \geq 0\right\}$, we have

\begin{equation*}
   D_2=\sum_{n \geq 0} \mathbb{E}\left[\left|F\left( \overline{\phi}\left(T-T_{n}, (\overline{y}_{n},v_{n})\right),v_{n}\right)-F\left(\phi\left(T-T_{n}, (y_{n},v_{n})\right),v_{n}\right)\right|^{2} \mathbbm{1}_{N_{T}=n} \mathbbm{1}_{\overline{N}_{T}=n} \mathbbm{1}_{\overline{\tau}^{\star}>n}\right]. 
\end{equation*}

\noindent Using Assumption \ref{Lipschitz_H} and Lemma \ref{betalem_1}, we get
\[
\lvert F\left( \overline{\phi}\left(T-T_{n}, (\overline{y}_{n},v_n)\right),v_{n}\right)-F\left( {\phi}\left(T-T_{n}, (y_{n},v_n)\right),v_{n}\right)\rvert ^2\leq C_{F} e^{L T}(n+1) C h .
\]
\noindent Then we have $D_2 \leq C_{3} h \sum_{n \geq 0} \mathbb{E}\left[\mathbbm{1}_{N_{T}=n}(n+1)\right]$, where $C_{3}$ is a constant independent of $h$. Using Assumption \ref{assumptionestimate}, we have 
\begin{equation*}
    \sum_{n \geq 0} \mathbb{E}\left[\mathbbm{1}_{N_{T}=n}(n+1)\right]=\mathbb{E}\left[\left(N_{T}+1\right)\right] \leq \mathbb{E}\left[\left(N_{T}^{*}+1\right)\right]<+\infty.
\end{equation*}
Subsequently, $D_2=O\left(h\right)$.
\end{proof}


\section{Weak error expansion}
\label{section5}
In this section, we aim to derive a weak error expansion for the PDifMP $(x_t)_{t\in[0,T]}$ and its corresponding Euler-Maruyama scheme $(\overline{x}_t)_{t\in[0,T]}$. To this end, we start by setting up the necessary framework.
\subsection{Setting}
\noindent Let $C^{m}(E_1)$, $m\in \mathbb{N}$, denote the space of functions defined on $E_1$ and possessing bounded and continuous derivatives with respect to $y$ of orders up to $m$ and let $C_{0}^{m}(E_1)$ be the space of functions in $C^{m}(E_1)$ with compact support in $E_1$. Further, let $C^{l,m}([0,T]\times E_1)$ represents the set of all real-valued functions defined on $[0,T]\times E_1$ such that they are $l$-times continuously differentiable in $t$ and $m$-times in $y$.\\
Consider the second order operator associated with the process $x_t=(y_t,v_t)$, acting on test functions $g$ of class $C^{2}(E_1)$. For all $x=(y,v)\in E$, $\mathcal{L}$ is given by
\begin{equation}
\label{Extended_gen}
	\mathcal{L}g(y,v)=b(y,v) \partial_y g(y,v)+\frac{1}{2}\sigma^2(y,v)\partial^2_yg(y,v)+\lambda(y,v)\int_{E}\left(g(y,\xi)-g(y,v)\right)\mathcal{Q}((y,v),d\xi).
	\end{equation}
Here, $\partial_y g(y,v)=\FTS{\partial g}{\partial y}(y,v)$ and $\partial^2_y g(y,v)=\FTS{\partial^2 g}{\partial{y^2}}(y,v)$. The detailed derivation of this operator can be found in \cite{bujorianu2006toward, bect2007processus, graham2013stochastic}. Further, for $A\in \mathcal{B}(E)$, we introduce the random counting measure $p(t,A)$ associated to $(x_t)_{t\in [0,T]}$, given by
\begin{equation*}
    p(t, A):= \sum_{n\geq 1} \mathbbm{1}_{(t\geq T_n)}\mathbbm{1}_{(x_{\scalebox{0.6}T_n}\in A)}.
\end{equation*}
Further, let $\Tilde{p}(t,A)$ be the compensator of $p(t,A)$, given by
\begin{equation*}
    \Tilde{p}(t,A)=\int_0^t\lambda(x_s)\mathcal{Q}(x_s,A)ds.
\end{equation*}
Consequently the compensated measure 
\begin{equation*}
q(t,A)=p(t,A)-\Tilde{p}(t,A)  
\end{equation*}
becomes an $\left(\mathcal{F}_t\right)_{t\in[0,T]}$-martingale generated by $p$, see \cite{yin2009hybrid, davis1984piecewise, yin2010properties}. Analogously, we define $\overline{p}, \Tilde{\overline{p}}, \overline{q}$ to be the same objects associated to the approximation $(\overline{x}_t)_{t\in[0,T]}$.\\ 
\noindent The extended generator \eqref{Extended_gen} leads to the generalised It\^o formula, a concept further detailed in \cite{bect2007processus, yin2010properties, graham2013stochastic}.\\
\noindent Before discussing the validity of the It\^o lemma, we first introduce the following notation.
    \begin{align*}
    	R^{(g)}_t:=&\int_0^t\int_E(g(s,\xi)-g(t,x_{s^{-}}))q(ds,d\xi),\\
    	L^{(g)}_t:=&\int_0^t \sigma(x_s) \partial_y g(s,x_s)dW_s,\\
     \overline{R}^{(g)}_t:=&\int_0^t\int_E(g(s,\xi)-g(t,\overline{x}_{s^{-}}))q(ds,d\xi),\\
    	\overline{L}^{(g)}_t:=&\int_0^t \sigma(\overline{x}_s)\partial_y g(s,\overline{x}_s)dW_s.
\end{align*}

\begin{definition}
\label{operators}
We define the following operators acting on functions $g\in C^{1,2}\left([0,T]\times E\right)$
\begin{align*}
\mathcal{T}g(t,x):=&\partial_t g(t,x)+b(x) \partial_yg(t,x)+\frac{1}{2}\sigma^2(x)\partial^2_yg(t,x),\\
\mathcal{R}g(t,x):=&\lambda(x)\int_{E}\big(g(t,\xi)-g(t,x)\big)\mathcal{Q}(x,d\xi),\\
\mathcal{A}g(t,x):=&\mathcal{T}g(t,x)+\mathcal{R}g(t,x),\\
\overline{\mathcal{T}}g(t,x,\xi):=&\partial_t g(t,x)+b(\xi) \partial_yg(t,x)+\frac{1}{2}\sigma^2(\xi)\partial^2_y g(t,x),\\
\overline{\mathcal{A}}g(t,x,\xi):=&\overline{\mathcal{T}}g(t,x,\xi)+\mathcal{R}g(t,x).
\end{align*}
\end{definition}
\noindent Note that for all functions $g\in C^{1,2}\left([0,T]\times E_1\right)$ it holds that $\overline{\mathcal{T}}g(t,x,x)=\mathcal{T}g(t,x)$. This implies that  $\overline{\mathcal{A}}g(t,x,x)=\mathcal{A}g(t,x)$.\\
To present the generalised It\^o formulas for the PDifMP $(x_t)_{t\in[0,T]}$ and its approximation $(\overline{x}_t)_{t\in[0,T]}$, we introduce the following notation.
\begin{notation}
For all $s\in [0,T]$, let $\overline{\eta}_s:= (\overline{T}_n+kh)\mathbbm{1}_{s\in \left[\overline{T}_n+kh,\, (\overline{T}_n+(k+1)h)\wedge \overline{T}_{n+1}\right)}$, where $n$ is a non-negative integer and $k\in \{ 0, \ldots \lfloor (\overline{T}_{n+1}-\overline{T}_n)/h\rfloor \}$.
\end{notation}
\begin{theorem}[Generalised It\^o formula]
Let $(x_t)_{t\in[0,T]}$ and $(\overline{x}_t)_{t\in[0,T]}$ be a PDifMP and its approximation respectively, with ${x_0=\overline{x}_0=x}$, for some $x\in E$. For all bounded functions $g\in C^{2}_0\left(E_1 \right)$ with bounded derivatives, the following hold
\begin{equation}
\label{Ito_pdif}
    g(t,x_t)-g(0,x)=\int_0^t \mathcal{A}g(s,x_s)ds +M^g_t, 
\end{equation}
where $M^g_t=R^{(g)}_t+L^{(g)}_t$ is a martingale, and
\begin{equation}
\label{Ito_bar}
    g(t,\overline{x}_t)-g(0,x)=\int_0^t \overline{\mathcal{A}}g(s,\overline{x}_s,\overline{x}_{\overline{\eta}_s})ds +\overline{M}^g_t, 
\end{equation}
where $\overline{M}^g_t=\overline{R}^{(g)}_t+\overline{L}^{(g)}_t$ is a martingale.
\end{theorem}
\begin{proof}
The proof for \eqref{Ito_pdif} can be found in \cite{bect2007processus,davis1984piecewise}. To prove the result for \eqref{Ito_bar} we proceed analogously. Let

\begin{equation*}
    \overline{M}_{t}^{g}=\sum_{k \geq 1} \mathbbm{1}_{t \geq \overline{T}_{k}}\left(g\left(\overline{T}_{k}, \overline{x}_{\overline{T}_{k}}\right)-g\left(\overline{T}_{k}, \overline{x}_{\overline{T}_{k}^{-}}\right)\right)-\int_{0}^{t} \mathcal{R} g\left(s, \overline{x}_{s}\right) ds.
\end{equation*}

\noindent As in \cite{davis1984piecewise}, on the event $\left\{\overline{N}_{t}=n\right\}$, we have

\begin{equation}
\label{ito_pr}
\footnotesize{
 \begin{aligned}
&\sum_{k \geq 1} \mathbbm{1}_{t\geq\overline{T}_{k}}\left(g\left(\overline{T}_{k}, \overline{x}_{\overline{T}_{k}}\right)-g\left(\overline{T}_{k}, \overline{x}_{\overline{T}_{k}^{-}}\right)\right) 
=\left[\sum_{k=0}^{n-1} g\left(\overline{T}_{k+1}, \overline{x}_{\overline{T}_{k+1}}\right)-g\left(\overline{T}_{k}, \overline{x}_{\overline{T}_{k}}\right)+g\left(t, \overline{x}_{t}\right)-g\left(\overline{T}_{n}, \overline{x}_{\overline{T}_{n}}\right)\right]\\&\hspace{6cm}-\left[\sum_{k=0}^{n-1} g\left(\overline{T}_{k+1}, \overline{x}_{\overline{T}_{k+1}^{-}}\right)-g\left(\overline{T}_{k}, \overline{x}_{\overline{T}_{k}}\right)+g\left(t, \overline{x}_{t}\right)-g\left(\overline{T}_{n}, \overline{x}_{\overline{T}_{n}}\right)\right].
\end{aligned}} 
\normalsize
\end{equation}

\noindent The first term on the r.h.s is equal to $g\left(t, \overline{x}_{t}\right)-g\left(0, x\right)$. As for the second term, we follow the approach employed in \cite{lemaire2020thinning} and decompose the increment $g\left(\overline{T}_{k+1}, \overline{x}_{\overline{T}_{k+1}^{-}}\right)-g\left(\overline{T}_{k}, \overline{x}_{\overline{T}_{k}}\right)$ into a sum of increments over the intervals ${\left[\overline{T}_{k}+i h,\left(\overline{T}_{k}+(i+1) h\right) \wedge \overline{T}_{k+1}\right] \subset\left[\overline{T}_{k}, \overline{T}_{k+1}\right]}$, for all $k \leq n-1$.\\
\noindent To simplify the discussion, we focus on increments of the form
\begin{equation*}
    g\big(t, \overline{\phi}(t, (y,v)),v\big)-g\big(i h, (\overline{y}_{i},v)\big), \qquad \forall \, x=(y,v)\in E, \quad t \in[i h,(i+1) h],
\end{equation*}
\noindent for some $i \geq 0,$ where $\overline{\phi}$ is defined in $\eqref{varphi_bar}$. Then, for $g\in C^{1,2}\left( [0,T]\times E_1 \right)$, we write
\footnotesize{
\begin{equation*}
 g\left(t, \overline{\phi}(t, (y,v)),v\right)-g\left(i h, (\overline{y}_{i},v),v\right)=\int_{i h}^{t}\left(\partial_{t} g+b\left(\overline{y}_{i},v\right) \partial_{y} g+\frac{1}{2}\sigma^2\left(\overline{y}_{i},v\right)\partial^2_{y}g\right)\left(s,\overline{\phi}\big(s, (y,v)\big),v\right) ds.
\end{equation*}}
\normalsize
\noindent Then, the above arguments together with Definition \ref{operators} allow us to express the second term on the r.h.s of \eqref{ito_pr} as
\begin{equation*}
\sum_{k=0}^{n-1} g\left(\overline{T}_{k+1}, \overline{x}_{\overline{T}_{k+1}^{-}}\right)-g\left(\overline{T}_{k}, \overline{x}_{\overline{T}_{k}}\right)+g\left(t, \overline{x}_{t}\right)-g\left(\overline{T}_{n}, \overline{x}_{\overline{T}_{n}}\right)=\int_{0}^{t} \overline{\mathcal{T}} g\left(s, \overline{x}_{s}, \overline{x}_{\overline{\eta}_s}\right) d s .    
\end{equation*}
\end{proof}

\subsection{Weak error analysis}
Consider the operator $\mathcal{A}$ defined in Definition \ref{operators} and the parabolic partial differential equation (PDE) with a terminal condition 
\begin{equation}
   \label{integro}
        \left\{
    \begin{array}{ll}
\mathcal{A}u(t,x)= 0,& \quad (t,x)\in[0,T)\times E,\\
 u(T,x) = F(x), &\quad x\in E.
\end{array}
\right. 
\end{equation}

\noindent Under the following assumptions the problem \eqref{integro} has a unique solution which is continuous on $[0,T]\times E$ and belongs to the space $C^{1,2}([0,T)\times E_1)$.\\
Let $x=(y,v)\in E$ and let $D^{\alpha}_y$ be a vector representing the collection of all partial derivatives of order up to $\alpha$, where $\alpha \leq m$, $m\mathbb{N}$, given by
\begin{equation*}
    D^{\alpha}_y=\FTS{\partial^{\alpha}}{\partial y^{\alpha}}.
\end{equation*}
\begin{assumption}
\label{reg_assump}
Let $x=(y,v)\in E$.
\begin{enumerate}
    \item The function $F(x)\in C^{2}(E_1)$ and is bounded. Further, the growth of $F$ is at most polynomial.
    \item For all $v \in \mathbf{V}$ and for all $B\in \mathcal{B}(\mathbf{V})$, the functions $y\mapsto Q((y,v),B)$, $y\mapsto \lambda (y,v)$, ${y\mapsto b(y,v)}$ and $y\mapsto \sigma (y,v)$ are all bounded and twice continuously differentiable with bounded derivatives.
\end{enumerate}

\end{assumption}

\begin{assumption}
 For each $v\in \mathbf{V}$, the functions $b$ and $\sigma$ are both twice continuously differentiable with respect to $y$ and that they satisfy Theorem \ref{existence}. Further, assume that there exist two constants $K$ and $r\in \mathbb{N}$, such that
     \begin{equation*}
        \vert D^{\alpha}_yb(y,v)\vert +  \vert D^{\alpha}_y\sigma(y,v)\vert \leq K(1+\vert y \vert^{r}), \quad \forall y\in E_1 \quad \text{and} \quad \alpha \in \{0,\ldots 4\}.
    \end{equation*}
    \label{Diff_ass}
\end{assumption}
\begin{theorem}[Feynman-Kac formula]
\label{Fey_Kac}
Let $x=(y,v)\in E$. Suppose that Assumptions \ref{reg_assump} and \ref{Diff_ass} are satisfied. Suppose also that the PDE \eqref{integro} admits a unique solution $u(t,x)$ of class ${C^{1,2}([0,T)\times E_1)}$. Additionally, if the mapping $y\mapsto u(t, (y,v))$ is bounded and its partial derivatives of all orders are bounded and uniformly Lipschitz continuous in $v$. Then, $u(t,x)$ is given by
\begin{equation}
\label{Feynman}
    u(t,x)=\mathbb{E}[F(x_T)\vert x_t=x], \quad \forall \, t\in [0,T], \, x\in E.
\end{equation}
\end{theorem}

\begin{proof}
We refer to \cite{yin2010properties} for the proof.
\end{proof}

\begin{assumption}
\label{reg_assump2}
Let $x=(y,v)\in E$.
\begin{enumerate}
    \item The functions $b$ and $\sigma$ are of class $C^{\infty}(E_1)$, and their partial derivatives of all orders are bounded. 
    \item The function $F$ is of class $C^{\infty}(E_1)$ and for any $\alpha \leq m$, $m\in \mathbb{N}$, there exists an integer $r$ and a positive constant $K$ such that
\begin{equation*}
        \vert D^{\alpha}_yF(y,v)\vert \leq K(1+\vert y \vert^{r}), \quad \forall y\in E_1 \, \quad \text{and} \quad \alpha \in \{0,\ldots 4\}.
    \end{equation*}
\end{enumerate}

\end{assumption}

\begin{theorem}
\label{Weak_conv}
For $t \in[0,T]$, let $\left(x_{t}\right)$ be a PDifMP and $\left(\overline{x}_{t}\right)$ its approximation with ${x_{0}=\overline{x}_{0}=x}$ for some $x \in E$. Assume the conditions of Theorem \ref{Fey_Kac}. Under Assumptions \ref{reg_assump2} for any $F: E \rightarrow \mathbb{R}$ the discretisation error of the Euler-Maruyama scheme satisfies
\begin{equation*}
  \mathbb{E}\left[F\left(\overline{x}_{T}\right)\right]-\mathbb{E}\left[F\left(x_{T}\right)\right]=\frac{h}{2}  \mathbb{E}\left[\int_{0}^{T} \gamma\left(s, x_{s}\right) d s\right] +O\left(h^{2}\right),
\end{equation*}
\noindent where

\begin{equation*}
  \gamma\left(t,\,\cdot\,\right)=\left(  \overline{\mathcal{A}}(\partial_tu)+b\overline{\mathcal{A}}(\partial_yu)+\frac{1}{2}\sigma^2\overline{\mathcal{A}}(\partial_y^2u)+\overline{\mathcal{A}}(\mathcal{R}u)\right)(t,\,\cdot\,).
\end{equation*}

\end{theorem}
\begin{proof}
The proof of Theorem \ref{Weak_conv} is inspired from \cite{lemaire2020thinning, graham2013stochastic, pages2018numerical, talay1990expansion} and it involves two primary steps. First derive a representation of the weak error. Then, we use this representation to identify $\gamma(t,\cdot)$.\\

\noindent \textit{Step 1: Representing} $\mathbb{E}[F(\overline{x}_T)-F(x_t)]$.\\

\noindent It follows from the Feynman-Kac formula (\ref{Feynman}), the terminal condition $u(T,x)=F(x)$ and the application of the generalised It\^o formula (\ref{Ito_bar}) to $u$ at time $T$ that
\begin{equation*}
    u\left(T, \overline{x}_{T}\right)=u(0, x)+\int_{0}^{T} \overline{\mathcal{A}} u\left(s, \overline{x}_{s}, \overline{x}_{\overline{\eta}_s}\right) d s+\overline{M}_{T}^{u},
\end{equation*}
where $\overline{M}_{T}^{u}=\overline{R}_{T}^{u}+\overline{L}_{T}^{u}$. Since $\overline{M}_{T}^{u}$ is an $\mathcal{F}_t$-martingale, it has zero expectation, \cite{pages2018numerical,baran2013feynman}. Thus,
\begin{equation*}
    \mathbb{E}\left[u\left(T, \overline{x}_{T}\right)-u(0, x)\right]=\mathbb{E}\left[\int_{0}^{T} \overline{\mathcal{A}} u\left(s, \overline{x}_{s}, \overline{x}_{\overline{\eta}_s}\right) d s\right],
\end{equation*}
where
\begin{equation}
\label{A_bar}
   \overline{\mathcal{A}} u\left(s, \overline{x}_{s}, \overline{x}_{\overline{\eta}_s} \right)=\partial_t u(s,\overline{x}_{s})+b(\overline{x}_{\overline{\eta}_s}) \partial_yu(s,\overline{x}_{s})+\frac{1}{2}\sigma^2(\overline{x}_{\overline{\eta}_s})\partial^2_yu(s,\overline{x}_{s})+\mathcal{R}u(s,\overline{x}_s).
\end{equation}
Given the regularity conditions on $\lambda,\, \mathcal{Q}$ and $u$ (see Assumptions \ref{reg_assump}-\ref{reg_assump2} and Theorem \ref{Fey_Kac}), the functions $\partial_{t} u, \, \partial_{y} u$ and $\mathcal{R} u$ are smooth enough to apply the It\^o formula \eqref{Ito_bar} between $\overline{\eta}_s$ and $s$, respectively, \cite{lemaire2020thinning}. We investigate each term separately
 \begin{equation*}
     \begin{gathered}
\partial_{t} u\left(s, \overline{x}_{s}\right)=\partial_{t} u\left(\overline{\eta}_s, \overline{x}_{\overline{\eta}_s}\right)+\int_{\overline{\eta}_s}^{s} \overline{\mathcal{A}}\left(\partial_{t} u\right)\left(r, \overline{x}_{r}, \overline{x}_{\overline{\eta}_r}\right) d r+\overline{M}_{s}^{\partial_{t} u}-\overline{M}_{\overline{\eta}_s}^{\partial_{t} u}, \\
\partial_{y} u\left(s, \overline{x}_{s}\right)=\partial_{y} u\left(\overline{\eta}_s, \overline{x}_{\overline{\eta}_s}\right)+\int_{\overline{\eta}_s}^{s} \overline{\mathcal{A}}\left(\partial_{y} u\right)\left(r, \overline{x}_{r}, \overline{x}_{\overline{\eta}_r}\right) d r+\overline{M}_{s}^{\partial_{y} u}-\overline{M}_{\overline{\eta}_s}^{\partial_{y} u}, \\
\partial_{y}^2 u\left(s, \overline{x}_{s}\right)=\partial_{y}^2 u\left(\overline{\eta}_s, \overline{x}_{\overline{\eta}_s}\right)+\int_{\overline{\eta}_s}^{s} \overline{\mathcal{A}}\left(\partial_{y}^2 u\right)\left(r, \overline{x}_{r}, \overline{x}_{\overline{\eta}_r}\right) d r+\overline{M}_{s}^{\partial_{y}^2 u}-\overline{M}_{\overline{\eta}_s}^{\partial_{y}^2 u}, \\
\mathcal{R} u\left(s, \overline{x}_{s}\right)=\mathcal{R} u\left(\overline{\eta}_s, \overline{x}_{\overline{\eta}_s}\right)+\int_{\overline{\eta}_s}^{s} \overline{\mathcal{A}}(\mathcal{R} u)\left(r, \overline{x}_{r}, \overline{x}_{\overline{\eta}_r}\right) d r+\overline{M}_{s}^{\mathcal{R} u}-\overline{M}_{\overline{\eta}_s}^{\mathcal{R} u}.
\end{gathered}
 \end{equation*}
Further, given that $\overline{\eta}_r=\overline{\eta}_s$ for any $r\in [\overline{\eta}_s,s]$, we have

\begin{multline*}
    b(\overline{x}_{\overline{\eta}_s})\partial_{y} u\left(s, \overline{x}_{s}\right)=b(\overline{x}_{\overline{\eta}_s})\partial_{y} u\left(\overline{\eta}_s, \overline{x}_{\overline{\eta}_s}\right)+\int_{\overline{\eta}_s}^{s}b(\overline{x}_{\overline{\eta}_r}) \overline{\mathcal{A}}\left(\partial_{y} u\right)\left(r, \overline{x}_{r}, \overline{x}_{\overline{\eta}_r}\right) d r\\+b(\overline{x}_{\overline{\eta}_s})\left(\overline{M}_{s}^{\partial_{y} u}-\overline{M}_{\overline{\eta}_s}^{\partial_{y} u}\right),
\end{multline*}
and
\begin{multline*}
    \frac{1}{2}\sigma^2\partial_{y}^2(\overline{x}_{\overline{\eta}_s}) u\left(s, \overline{x}_{s}\right)=\frac{1}{2}\sigma^2\partial_{y}^2 u\left(\overline{\eta}_s, \overline{x}_{\overline{\eta}_s}\right)+\frac{1}{2}\int_{\overline{\eta}_s}^{s}\sigma^2 (\overline{x}_{\overline{\eta}_r}) \overline{\mathcal{A}}\left(\partial_{y}^2 u\right)\left(r, \overline{x}_{r}, \overline{x}_{\overline{\eta}_r}\right) d r\\+\frac{1}{2}\sigma^2 (\overline{x}_{\overline{\eta}_s})\left(\overline{M}_{s}^{\partial_{y}^2 u}-\overline{M}_{\overline{\eta}_s}^{\partial_{y}^2 u}\right),
\end{multline*}
so that (\ref{A_bar}) is written as

\begin{multline}
\label{A_bar_new}
\overline{\mathcal{A}} u\left(s, \overline{x}_{s}, \overline{x}_{\overline{\eta}_s}\right) =\overline{\mathcal{A}} u\left(\overline{\eta}_s, \overline{x}_{\overline{\eta}_s}, \overline{x}_{\overline{\eta}_s}\right)+\int_{\overline{\eta}_s}^{s} \Gamma\left(r, \overline{x}_{r}, \overline{x}_{\overline{\eta}_r}\right) d r 
+\left(\overline{M}_{s}^{\partial_{t} u}-\overline{M}_{\overline{\eta}_s}^{\partial_{t} u}\right)\\+
b(\overline{x}_{\overline{\eta}_s})\left(\overline{M}_{s}^{\partial_{y} u}-\overline{M}_{\overline{\eta}_s}^{\partial_{y} u}\right)+ \frac{1}{2}\sigma^2 (\overline{x}_{\overline{\eta}_s})\left(\overline{M}_{s}^{\partial_{y}^2 u}-\overline{M}_{\overline{\eta}_s}^{\partial_{y}^2 u}\right)+\\
+\left(\overline{M}_{s}^{\mathcal{R} u}-\overline{M}_{\overline{\eta}_s}^{\mathcal{R} u}\right)),
\end{multline}

\noindent where

\begin{equation}
\label{Gamma}
  \Gamma\left(t, x, \xi\right)=\left(  \overline{\mathcal{A}}(\partial_tu)+b(\xi)\overline{\mathcal{A}}(\partial_yu)+\frac{1}{2}\sigma^2(\xi)\overline{\mathcal{A}}(\partial_y^2u)+\overline{\mathcal{A}}(\mathcal{R}u)\right)(t,x,\xi).
\end{equation}

\noindent Using Theorem \ref{Fey_Kac} and the fact that $\overline{\mathcal{A}}(t,x,x)=\mathcal{A}(t,x)$, the first term in (\ref{A_bar_new}) vanishes. Further, since $(\overline{M}_{s}^{\partial_{t} u})$, $(\overline{M}_{s}^{\partial_{y} u})$ and $(\overline{M}_{s}^{\partial_{y}^2 u})$ are martingales, see \cite{pages2018numerical}. Moreover, using Fubini's theorem, it is straightforward to see that they have zero expectations. Therefore, the original expansion is reduced to
\begin{equation}
\label{Original_exp}
    \mathbb{E}\left[F\left(\overline{x}_{T}\right)\right]-\mathbb{E}\left[F\left(x_{T}\right)\right]=\mathbb{E}\left[\int_{0}^{T} \int_{\overline{\eta}_s}^{s} \Gamma\left(r, \overline{x}_{r}, \overline{x}_{\overline{\eta}_r}\right) d r d s\right].
\end{equation}
The function $\Gamma$ can be written explicitly in terms of $u$, $b$, $\sigma$, $\lambda$, $\mathcal{Q}$ and their partial derivatives at $t$ and $y$. Considering the terms of (\ref{Gamma}) separately, we have

\begin{equation*}
\small{
    \begin{aligned}
&\overline{\mathcal{A}}\left(\partial_{t} u\right)(t, x, \xi)=\partial_{t}^{2} u(t, x)+b(\xi) \partial_{t y}^{2} u(t, x)+ \frac{1}{2}\sigma^2b(\xi)\partial_y^2u(t,x)  + \mathcal{R}\left(\partial_{t} u\right)(t, x), \\
&\left(b \overline{\mathcal{A}}\left(\partial_{y} u\right)\right)(t, x, \xi)=b(\xi)\left(\partial_{ty}^{2} u(t, x)+b(\xi) \partial_{y}^{2} u(t, x)+ \frac{1}{2}\sigma^2b(\xi)\partial_y^3u(t,x) + \mathcal{R}\left(\partial_{y} u\right)(t, x)\right), \\
&\left(\frac{1}{2}\sigma^2 \overline{\mathcal{A}}\left(\partial_{y}^2 u\right)\right)(t, x, \xi)=\frac{1}{2}\sigma^2b(\xi)\left(\partial_{t}\partial_{y}^{2} u(t, x)+b(\xi) \partial_{y}^{3} u(t, x)+ \frac{1}{2}\sigma^2b(\xi)\partial_y^4u(t,x) + \mathcal{R}\left(\partial_{y}^2 u\right)(t, x)\right), \\
&\overline{\mathcal{A}}(\mathcal{R} u)(t, x, \xi)=\partial_{t}(\mathcal{R} u)(t, x)+b(\xi) \partial_{y}(\mathcal{R} u)(t, x)+ \frac{1}{2}\sigma^2b(\xi)\partial_y^2(\mathcal{R}u)(t,x) + \mathcal{R}(\mathcal{R} u)(t, x),
\end{aligned}}
\end{equation*}
where $\partial_{t}^{2} u=\FTS{\partial^{2}u}{\partial {t^2}}$, $\partial_{ty}^{2} u=\FTS{\partial^{2}u}{\partial t \partial y }$, $\partial_{t}^{3} u=\FTS{\partial^{3}u}{\partial {y^3}}$, and $\partial_{t}^{4} u=\FTS{\partial^{4}u}{\partial {y^4}}$.\\

\noindent To approximate the functions $\partial_{t}^{2} u$, $\partial_{ty}^{2} u$, $\partial_{y}^{2} u$, $\partial_{y}^{3} u$, $\partial_t \partial_{y}^{2} u$, $\partial_{y}^{4} u$, $\mathcal{R}\left(\partial_{t} u\right)$, $\mathcal{R}\left(\partial_{y} u\right)$, $\mathcal{R}\left(\partial_{y}^2 u\right)$, $\partial_{t}(\mathcal{R} u)$, $\partial_{y}(\mathcal{R} u)$ and $\mathcal{R}(\mathcal{R} u)$ at the order 0 around $\left(\overline{\eta}_r, \overline{x}_{\overline{\eta}_r}\right)$, we employ the the Taylor formula as seen in \cite{graham2013stochastic, lemaire2020thinning}. This leads to
\begin{equation*}
    \Gamma\left(r, \overline{x}_{r}, \overline{x}_{\overline{\eta}_r}\right)=\Gamma\left(\overline{\eta}_r, \overline{x}_{\overline{\eta}_r}, \overline{x}_{\overline{\eta}_r}\right)+O\left(h\right).
\end{equation*}

\noindent Recalling the identity (\ref{Original_exp}), we can write the r.h.s of (\ref{Original_exp}) as
\begin{equation*}
    \begin{aligned}
    \mathbb{E}\left[\int_{0}^{T} \int_{\overline{\eta}_s}^{s} \Gamma\left(r, \overline{x}_{r}, \overline{x}_{\overline{\eta}_r}\right) d r d s\right]=& \,\mathbb{E}\left[\int_{\overline{\eta}_s}^{s}d s\int_{0}^{T}  \Gamma\left(r, \overline{x}_{r}, \overline{x}_{\overline{\eta}_r}\right) d r d s\right],\\
    =&\, \mathbb{E}\left[(\overline{\eta}_s-{s})\int_{0}^{T}  \Gamma\left(r, \overline{x}_{r}, \overline{x}_{\overline{\eta}_r}\right) d r d s\right].
    \end{aligned}
\end{equation*}
\noindent We associate to the function $\Gamma$ the function $\gamma\in C^{1,2}\left( [0,T]\times E_1\right)$ defined by
\begin{equation*}
 \gamma(t, x)=\Gamma(t, x, x).
\end{equation*}
\noindent Further, for any $r \in[\overline{\eta}_s, s]$ it holds that $ \overline{\eta}_r=\overline{\eta}_s$ and that $|s-\overline{\eta}_s| \leq h$. Thus, we obtain
\begin{equation}
\label{Expect_decomp}
   \mathbb{E}\left[\int_{0}^{T} \int_{\overline{\eta}_s}^{s} \Gamma\left(r, \overline{x}_{r}, \overline{x}_{\overline{\eta}_r}\right) d r d s\right]=\mathbb{E}\left[\int_{0}^{T}(s-\overline{\eta}_s) \gamma\left(\overline{\eta}_s, \overline{x}_{\overline{\eta}_s}\right) d s\right].
\end{equation}

\noindent We now decompose the integral in (\ref{Expect_decomp}) into a finite sum of integrals over intervals of the form $\left[\overline{T}_{n}+k h,\left(\overline{T}_{n}+(k+1) h\right) \wedge \overline{T}_{n+1}\right]$. In these intervals, the function $\gamma$ is assumed to be constant, and hence, we only consider integrals of the form $\int_{kh}^{t}(s-k h) C d s$, where $k \geq 0$, $t \in[k h,(k+1) h]$ and $C$ is bounded. Thus,
\begin{equation*}
  \int_{k h}^{t}(s-k h) C d s=\frac{t-k h}{2} \int_{k h}^{t} C d s.  
\end{equation*}

\noindent If we write $t-kh=(t-(k+1)h) +h$, we get
\begin{equation*}
  \int_{kh}^{t}(s-kh) C d s=\frac{t-(k+1) h}{2} \int_{kh}^{t} C d s+\frac{h}{2} \int_{kh}^{t} C d s.
  \end{equation*}

\noindent Given that $C$ is a bounded constant, we deduce that 
\begin{equation*}
    \int_{k h}^{t}(s-kh) C d s=\frac{h}{2} \int_{k h}^{t} C d s+O\left(h^{2}\right).
\end{equation*}
\noindent Combining these results with the fact that $\gamma$ is assumed bounded and that ${\mathbb{E}\left[\overline{N}_{T}\right]<+\infty}$, we obtain the following representation
 
 \begin{equation}
 \label{gamma_eta}
    \mathbb{E}\left[F\left(\overline{x}_{T}\right)\right]-\mathbb{E}\left[F\left(x_{T}\right)\right]=\frac{h}{2} \mathbb{E}\left[\int_{0}^{T} \gamma\left(\overline{\eta}_s, \overline{x}_{\overline{\eta}_s}\right) d s\right]+O\left(h^{2}\right) . 
 \end{equation}

\noindent \textit{Step 2: First order expansion.}\\

We now consider the first term in the r.h.s of (\ref{gamma_eta}) and introduce the following random variables
\begin{equation*}
    \chi:=\int_{0}^{T} \gamma\left(\overline{\eta}_s, x_{\overline{\eta}_s}\right) d s  \qquad \text{and} \qquad \overline{\chi}:=\int_{0}^{T} \gamma\left(\overline{\eta}_s, \overline{x}_{\overline{\eta}_s}\right) d s.
\end{equation*}
Recalling $\overline{\tau}^{\star}$ from Definition \ref{taudef}, we write
\begin{equation}
\label{Exp_chi}
  \mathbb{E}[|\overline{\chi}-\chi|]=\mathbb{E}\left[|\overline{\chi}-\chi|\mathbbm{1}_{\min \left(T_{\overline{\tau}^{\star}}, \overline{T}_{\overline{\tau}^{\star}}\right) \leq T}\right]+\mathbb{E}\left[|\overline{\chi}-\chi|\mathbbm{1}_{\min \left(T_{\overline{\tau}^{\star}}, \overline{T}_{\overline{\tau}^{\star}}\right)>T}\right].  
\end{equation}
Since $\gamma$ is bounded and $\mathbb{P}\left(\min \left(T_{\overline{\tau}^{\star}}, \overline{T}_{\overline{\tau}^{\star}}\right) \leq T\right)=O\left(h\right)$ (as shown in the proof of Theorem \ref{Strong_thm}), we have
\begin{equation*}
\mathbb{E}\left[|\overline{\chi}-\chi| \mathbbm{1}_{\min \left(T_{\overline{\tau}^{\star}}, \overline{T}_{\overline{\tau}^{\star}}\right) \leq T}\right]=O\left(h\right).    
\end{equation*}
To address the second term on the r.h.s of (\ref{Exp_chi}), recall from (\ref{T_events}) that, on the event $\left\{\min \left(T_{\overline{\tau}^{\star}}, \overline{T}_{\overline{\tau}^{\star}}\right)>T\right\}$, the trajectories of the discrete time processes $(T_n,v_n)$ and $(\overline{T}_n, \overline{v}_n)$ are equal for all $n$ such that $T_n\in [0,T]$ (or equivalently $\overline{T}_n \in [0,T]$). This implies that $T_{k}=\overline{T}_{k}$ and $v_{k}=\overline{v}_{k}$ for all $k \geq 1$ such that $T_{k} \in[0, T]$. Thus, for all $n \leq \overline{N}_{T}$ and for all $s \in\left[\overline{T}_{n}, \overline{T}_{n+1}\left)\right.\right.$, we have

\begin{equation*}
x_{\overline{\eta}_s}=\left(( \phi\left((\overline{\eta}_s-\overline{T}_{n}), \overline{x}_{n}\right),v_n\right) \qquad \text{and} \qquad \overline{x}_{\overline{\eta}_s}=\left(( \overline{\phi}((\overline{\eta}_s-\overline{T}_{n}), \overline{x}_{n})),\overline{v}_n\right).
\end{equation*}

\noindent Therefore, on the event $\left\{\min \left(T_{\overline{\tau}^{\star}}, \overline{T}_{\overline{\tau}^{\star}}\right)>T\right\}$ we have
\begin{equation*}
   |\overline{\chi}-\chi| \leq \sum_{n=0}^{\overline{N}_{T}} \int_{\overline{T}_{n}}^{\overline{T}_{n+1} \wedge T}\left|\gamma\left(\overline{\eta}_s, \overline{x}_{\overline{\eta}_s}\right)-\gamma\left(\overline{\eta}_s, x_{\overline{\eta}_s} \right)\right| d s . 
\end{equation*}

\noindent Since $\gamma$ is defined as sum and product of bounded Lipschitz continuous functions, it is straightforward to see that under Assumptions \ref{reg_assump}-\ref{reg_assump2} and the conditions of Theorem \ref{Fey_Kac}, the mapping $y \mapsto \gamma(t, (y,v))$ is uniformly Lipschitz continuous in $(t, v)$ with constant $L_{\gamma}$. Using this Lipschitz property and applying Lemma \ref{betalem_1}, we obtain

\begin{equation*}
  \left|\gamma\left(\overline{\eta}_s, \overline{x}_{\overline{\eta}_s}\right)-\gamma\left(\overline{\eta}_s, x_{\overline{\eta}_s} \right)\right|  \leq L_{\gamma} C e^{L T}(n+1) h . 
\end{equation*}
Therefore, we can write
\begin{equation*}
  \mathbb{E}\left[|\overline{\chi}-\chi|\mathbbm{1}_{\min \left(T_{\overline{\tau}^{\star}}, \overline{T}_{\overline{\tau}^{\star}}\right)>T}\right] \leq L_{\gamma} C e^{L T} T h \mathbb{E}\left[\overline{N}_{T}\left(\overline{N}_{T}+1\right)\right].
\end{equation*}

\noindent Furthermore, given that $\overline{N}_{T} \leq N_{T}^{*}$ and $\mathbb{E}\left[N_{T}^{*}\left(N_{T}^{*}+1\right)\right]<+\infty$, it follows that

\begin{equation*}
   \mathbb{E}\left[|\overline{\chi}-\chi|\mathbbm{1}_{\min \left(T_{T^{\star}}, \overline{T}_{T^{\star}}\right)>T}\right]=O\left(h\right).
\end{equation*}
Therefore,
\begin{equation}
 \mathbb{E}\left[\int_{0}^{T} \gamma\left(\overline{\eta}_s, \overline{x}_{\overline{\eta}_s}\right) d s\right]=\mathbb{E}\left[\int_{0}^{T} \gamma\left(\overline{\eta}_s, x_{\overline{\eta}_s}\right) d s\right]+O\left(h\right).   
\end{equation}

\noindent In addition, under the regularity Assumptions \ref{reg_assump}-\ref{reg_assump2} and by Theorem \ref{Fey_Kac}, the function $(t, y) \mapsto \gamma(t,(y,v))$  exhibits uniform Lipschitz continuity in $v$. Further, for all $s \in[0, T]$ there exits $k \geq 0$ such that both $s$ and $\overline{\eta}_s$ fall within the same interval $\left[\overline{T}_{k}, \overline{T}_{k+1}\left)\right.\right.$ so that we have
\begin{equation*}
x_{s}=\left(\phi\left((s-\overline{T}_{k}), x_{k}\right),v_k\right) \qquad \text{and} \qquad x_{\overline{\eta}_s}=\left( \phi\left((\overline{\eta}_s-\overline{T}_{k}), x_{k}\right),v_k\right).
\end{equation*}

\noindent Employing the Lipschitz continuity of $\gamma$, the fact that $|s-\overline{\eta}| \leq h$ and the uniform boundness of both $b$ and $\sigma$, we obtain
\begin{equation*}
    \left|\gamma\left(s, x_{s}\right)-\gamma\left(\overline{\eta}_s, x_{\overline{\eta}_s}\right)\right| \leq C h,
\end{equation*}
 where $C$ is a constant independent of $h$. Thus, we can further conclude that
 \begin{equation*}
     \sup _{s \in[0, T]}\left| \mathbb{E}\left[\gamma\left(s, x_{s}\right)\right]-\mathbb{E}\left[\gamma\left(\overline{\eta}_s, x_{\overline{\eta}_s}\right)\right]\right|  \leq C h.
 \end{equation*}
 
\noindent Consequently,
\begin{equation}
   \left|\mathbb{E}\left[\int_{0}^{T} \gamma\left(\overline{\eta}_s, x_{\overline{\eta}_s}\right) d s\right]-\mathbb{E}\left[\int_{0}^{T} \gamma\left(s, x_{s}\right) d s\right]\right| \leq C T h 
\end{equation}
 Finally, the weak error expansion reads
 \begin{equation*}
   \mathbb{E}\left[F\left(\overline{x}_{T}\right)\right]-\mathbb{E}\left[F\left(x_{T}\right)\right]=\frac{h}{2} \mathbb{E}\left[\int_{0}^{T} \gamma\left(s, x_{s}\right) d s\right]+O\left(h^{2}\right) .  
 \end{equation*}
\end{proof}
\section{Test results}
In this section, we present the simulation results of the TEM scheme for modelling PDifMPs. These experiments, designed as test cases, aim to fill the gap in the existing literature where such processes have not yet been simulated using this method. By implementing these simulations, we investigate the performance of the scheme and its mean-square convergence, providing new insights into stochastic systems characterised by both continuous dynamics and Poisson-driven jumps.\\
In our adaptive grid setup, where time steps are not equidistant, the error calculation needs to account for varying interval sizes. To quantify the accuracy of the TEM scheme on such a grid, we compute the root mean-square error (RMSE) across all simulated paths and intervals. This approach is essential for capturing the dynamic adjustments in the grid and ensuring that our error measurement reflects the true behaviour of the simulation under test conditions. The error at each time step is calculated using the following formula
\begin{equation}
    \text{err}_{h,M} = \max_{n=0, \dots, N_T}\left(\frac{1}{M} \sum_{j=1}^M \big| x(t_n, \omega_j) - \Tilde{x}(t_n, \omega_j) \big|^2\right)^{1/2},
\end{equation}

where $M$ is the number of simulated paths, $ N_T $ is the total number of intervals in the adaptive grid for each simulation, dynamically determined by the process characteristics and the occurrence of jumps, here $T$ is a fixed final time. In this context, $ t_n $ are the points in time at which the state of the system is recorded, spaced irregularly due to the adaptive grid approach. These points are defined for each interval $[T_n, T_{n+1}]$ as
\begin{equation*}
        \{T_n + k h_n\}, \quad k = 0, \dots, \left\lfloor \frac{T_{n+1} \wedge T - T_n}{h_n} \right\rfloor,
\end{equation*}

where $ h_n $ is the local time step size for the interval, calculated as $ h_n = \frac{T_{n+1} \wedge T - T_n}{\lfloor (T_{n+1} \wedge T - T_n) / h \rfloor} $, ensuring that the grid points align with the end of each interval. In our computations we used $M=200$.\\
The results are presented as figures, where we have plotted the accuracy versus the step sizes on a logarithmic scale with a base $2$. The slope of the resulting lines corresponds to the observed order of the scheme. Lines with slopes of $0.5$ and $1$ are shown to allow comparisons with the convergence of these orders.
    
\subsection{Example 1: Geometric Brownian motion with Poisson jumps}
\label{Sim_Example1}
Our first example considers a PDifMP, $x_t=(y_t,v_t)$, where the continuous component, $y_t$, follows a geometric Brownian motion such that

\begin{equation}
\label{geom}
    dy_t=\mu y_t dt+\sigma y_t dW_t,
\end{equation}

on the time interval $[0,1]$. Here, $\mu$ represents the drift coefficient, $\sigma$ is the volatility coefficient, and $W_t$ is a standard Wiener process. The exact solution of \eqref{geom} is known and is given by

\begin{equation}
    y_t=y_0\exp{\big( (\mu-\frac{\sigma^2}{2})t+\sigma W_t \big)},
\label{sol_geo}
\end{equation}

where $y_0$ is the initial value of $y_t$ at $t = 0$.\\
 Further, the system experiences jumps at random times such that the discrete component $ v_t $ tracks the number of jumps that the system has experienced up to time $ t $. It is defined as an integer-valued counting process
\[
v_t = \sum_{i=1}^{N_t} 1,
\]
where $ N_t $ denotes the number of jumps by time $ t $, governed by a Poisson process with a fixed rate $ \lambda $. \\
The jump law $ \mathcal{Q} $ governs the stochastic evolution of the entire state $ x_t $ upon jumps. While $ v_t $ inherently tracks the occurrence of jumps, the continuous component $ y_t $ undergoes a multiplicative adjustment. The probabilistic rules for this update are specified by $ \mathcal{Q}$, which dictates how $ y_t $ is modified based on the jump magnitude $ \eta $, exponentially distributed with rate $ \lambda $, i.e, $\eta \sim \exp(\lambda)$, where the transition in the state from $ x $ to $ x' = (y', v') $ upon a jump is modeled as
\begin{equation*}
  y' = y \cdot e^\eta,  
\end{equation*}
and $ v' = v + 1 $, emphasizing that each jump increments the jump counter $ v_t $ by one. The corresponding procedure for simulating the stochastic dynamics of the PDifMP $x_t = (y_t, v_t)$ is detailed in Algorithm \ref{Alg_1} (see \ref{App_C}).\\
In Figure \ref{Example1} we present the convergence behaviour of the TEM scheme, which shows that the error decreases significantly as the step size is reduced, closely approximating a theoretical slope of $\frac{1}{2}$, indicating first-order convergence.

\begin{figure}[h!]
\centering
\includegraphics[width=0.55\textwidth]{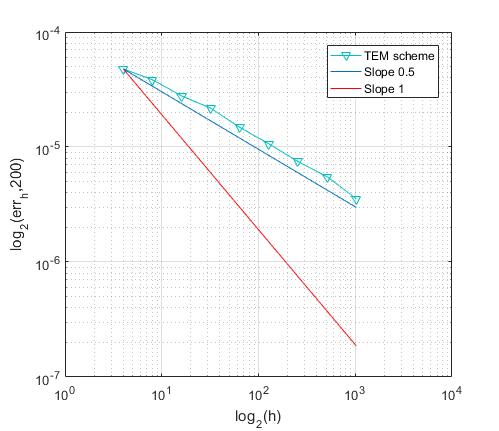}
\caption{Logarithm (base 2) of the (RMSE) of TEM scheme against the logarithm (base 2) of different values of the step size $h$. Parameters: final time $T=1$, initial value $y_0=50$, drift coefficient $\mu=0.001$, volatility coefficient $\sigma=0.002$, and Poisson jump rate $\lambda=0.0001$.}
\label{Example1}
\end{figure}

\subsection{Example 2: Geometric Brownian motion with adaptive jump rate function}

In our second example, we continue to use the geometric Brownian motion framework, as described by Equation \eqref{geom}, which was established in the first example. This time, however, we introduce a variation in the jump dynamics by implementing a non-homogeneous Poisson process, where the jump rate is no longer fixed but varies with the state of the process. The jump rate in this example varies with the state of the process, making it dependent on the current value of $y_t $. Specifically, the jump rate function $ \lambda(x_t)$ is given by:
\begin{equation*}
    \lambda(x_t):=\lambda(y_t,v_t) = 0.01 \times y_t,
\end{equation*}
where $y_t$ represents the current value of the continuous component. This formulation implies that higher values of the process lead to an increased likelihood of jumps, introducing a realistic dynamic that models increased market volatility or risk as asset values rise.\\
The objective of this simulation is to assess how the TEM scheme handles variable jump intensities that depend on the state of the process, particularly examining its impact on the accuracy and stability of the numerical solution. For the second example, where the jump rate function is adaptive and state-dependent, the transition in the state from $ x $ to $ x' = (y', v') $ upon a jump is defined by the simple yet significant modification
\begin{equation*}
    y^{'} = 0.9 \times y.
\end{equation*}
The  procedure for simulating the stochastic dynamics of the PDifMP $x_t = (y_t, v_t)$ is detailed in Algorithm \ref{Alg_2} (see \ref{App_C}).\\
\noindent Figure \ref{Example2} shows the convergence behaviour of the (TEM) scheme for a dynamic drift and volatility model with an adaptive jump rate function. The plot displays the logarithm (base 2) of the error against the logarithm (base 2) of different values of the step size $h$, illustrating the error reduction of the scheme as the step size decreases. Notably, the performance of the TEM scheme closely matches the theoretical slope of $\frac{1}{2}$, indicating first order convergence. This suggests that the TEM scheme effectively captures the dynamics of the model, even when the jump intensity varies with the state of the process, ensuring robust and accurate simulations over different step sizes.

\begin{figure}[h!]
\centering
\includegraphics[width=0.55\textwidth]{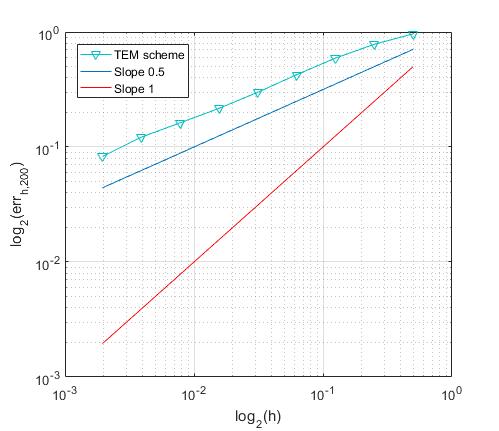}
\caption{Logarithm (base 2) of the (RMSE) of the TEM scheme against the logarithm (base 2) of different values of the step size $h$.  Parameters for this dynamic model include: final time $T=1$, initial value $y_0=50$, drift coefficient $\mu=0.01$, volatility coefficient $\sigma=0.2$, and thinning Poisson jump rate $\lambda^{*}=0.001$.}
\label{Example2}
\end{figure} 

The TEM scheme was effectively validated on two examples, demonstrating its accuracy and convergence. The first example provided a basic validation under controlled conditions, demonstrating the scheme's ability to handle stochastic jumps and replicate known solutions. The second example, involving state-dependent jump rates, introduced complexity and demonstrated the scheme's adaptability to dynamic changes in process parameters. Both examples confirmed the scheme's first-order convergence, which agrees well with theoretical results discussed in Section \ref{section4}.

\section{Numerical results: Simulation of glioma cell trajectories}
\label{section6}
We now apply the numerical method developed and analysed earlier to study the random movement of a single glioma cell at the microscopic scale, influenced by the tumour microenvironment. Our simulation approach involves employing the thinning method to generate jump times for the jump process (see Section \ref{Thinning_Sec}). Further, we use the Euler-Maruyama method to approximate the flow map $\phi$.\\
It would be ideal if we could compare the numerical solutions using our algorithms with the analytic solutions. Unfortunately, due to the complexity of the $x$-dependent switching process, closed-form solutions are not available. Instead, we use a splitting method to approximate the solution of \eqref{Micro_sys} between switching times for numerical demonstrations and comparison.
\subsection{PDifMP description for cell motion}
Let $u_t=(x_t,z_t,v_t)$, $t\in [0,T]$, be a stochastic process describing glioma cell movement, such that $x_t$ represents cell position, $z_t$ describes environmental signals influencing cell migration and $v_t$ refers to the cell velocity. The process $(u_t)_{t\in [0,T]}$ is a PDifMP, such that $(x_t,z_t)$ refer to the continuous component of the process, while $v_t$ is the jump process.
In the following analysis, we restrict to the one dimensional case for illustrative purposes, although the results can be generalised to two and three dimensions. We adopt a simplified model to assess the behaviour of the glioma cellular motion modelled through our simulation approach. We refer the interested readers to \cite{meddah2023stochastic, engwer2016multiscale, harpold2007evolution, swanson2000quantitative, painter2013mathematical, aubert2006cellular, hillen20042, hillen2013transport} and the references therein, for further details about various approaches for glioma modelling.\\
We define the state space of the model as $E=[-1,1]\times [0,1]\times \mathbf{V}_{\alpha\mathbb{S}}$, where $\mathbf{V}_{\alpha\mathbb{S}}$ represents a finite subset of $\alpha\mathbb{S}$. Here, $\mathbb{S}$ denotes the unit circle in $\mathbb{R}$, and $\alpha$, representing the mean speed of a tumour cell, is assumed to be constant. Since the brain is a bounded domain with dimensions approximately $140$ mm by $167$ mm \cite{Brain2021}, our simulation focuses on the local motion of a single glioma cell, and for computational and modeling simplicity we restrict the domain to $[-1,1]$. Further for the amount of bound receptors we restrict to $[0,1]$ as in \cite{meddah2023stochastic, conte2020glioma}.\\
The model describing a contact-mediated movement of a glioma cell at the microscopic scale reads

\begin{equation}
	\label{Micro_sys}
	\left\{
	\begin{array}{ll}
		dx_t & = \big(\frac{1}{2}z^2_tx_t-b z_tx_t+ az_tx_t+v_t\big)dt+\big(z_tx_t\big)dW_t,\\[0.2cm]
		dz_t & =-(k^{+}A(x)+k^{-})z_t+f^{'}(A(x))v_tA^{'}(x) dt,\\[0.2cm]
		dv_t & = 0dt.\\
	\end{array}
	\right. 
\end{equation}

\begin{itemize}
    \item Here, $W_t$ is a standard Wiener process and $k^{+}$ and $k^{-}$ are the rates of attachment and detachment between cell and tissue, respectively.
    \item The function $A(x)=\frac{1}{1+e^{-x}}$ represents the local concentration of extracellular factors that influence the binding process between cell surface receptors (integrins) and the microenvironment.
    \item  The function $f(A(x))=\frac{k^{+}A(x)}{k^{+}A(x)+k^{-}}$ represents the fraction of integrins bound to the extracellular matrix (ECM) at the steady state for a given local concentration $ A(x) $, see \cite{engwer2016multiscale, engwer2015glioma}.. Here, $A(x)$ models the concentration of extracellular factors that influence cell migration, ensuring a smooth transition between low and high concentrations.
    \item The parameter $z$ represents the amount of bound receptors of the cell surface to receptors in its extracellular environment. It affects how the cell interacts with the environment.
   \item The term $\frac{1}{2}z_t^2x_t$ represents the nonlinear response of the cell's motion to the concentration of bound receptors. This term captures the cell's directional movement influenced by the gradient of signaling molecules, potentially exhibiting saturation behaviour as receptor binding increases.
   \item The functions $az_tx_t$ and $-bz_tx_t$, $a,b>0$, represent the influence of external factors on the cell's movement. Specifically, $az_tx_t$ reflects the strength of a chemoattractant effect, which promotes cell migration towards higher concentrations of the attractant. Conversely, $-bz_tx_t$ represents a chemotactic repellent effect, discouraging cells from moving toward higher concentrations of the repellent when receptors are bound. These terms account for directional changes in response to the presence of signaling molecules, either attracting or repelling the cell's migration, depending on the nature of the molecules involved.
    \item The function $z_tx_t$ is the diffusion coefficient, representing random fluctuations in the cell's position.
\end{itemize}

\noindent In the analysis of the stochastic system \eqref{Micro_sys} governing glioma cell motion, we need to ensure that the coefficients satisfy the Lipschitz conditions. This is crucial for the stability and convergence of the numerical method. For the first SDE in \eqref{Micro_sys}, we specifically focus on the term \( \frac{1}{2}z^2_t x_t \) in the drift function, as the other terms are globally Lipschitz. To this end, we compute the partial derivatives with respect to \( x_t \) and \( z_t \):
\begin{align*}
    \frac{\partial }{\partial x_t}\left(\frac{1}{2}z^2_t x_t\right) &= \frac{1}{2}z^2_t, \\
    \frac{\partial}{\partial z_t}\left(\frac{1}{2}z^2_t x_t\right) &= z_t x_t.
\end{align*}
\noindent Given that \( z_t \in [0, 1] \) and \( x_t \in [-1, 1] \), these derivatives are bounded by \( \frac{1}{2} \) and 1, respectively. Therefore, the term \( \frac{1}{2}z^2_t x_t \) satisfies the local Lipschitz condition.\\
\noindent For the second ODE in \eqref{Micro_sys}, we analyse the following partial derivatives, the partial derivative with respect to \( x_t \) is given by
    \[
    -k^+ A^{'}(x) z_t + f^{''}(A(x)) A^{'}(x) v_t + f^{'}(A(x)) v_t A^{''}(x).
    \]
    Given that \( A(x)\) is a sigmoid function, \( A^{'}(x) \) and \( A^{''}(x) \) are bounded for \( x_t \in [-1, 1] \). It is straightforward to verify that the remaining terms in the drift functions and the diffusion coefficients also satisfy the necessary Lipschitz conditions.

\noindent The characteristic triplet $(\phi,\lambda,\mathcal{Q})$ of the PDifMP $u_t=(x_t,z_t,v_t)$ is given by

\small\begin{equation}
	\label{charc}
	\left\{
	\begin{array}{ll}
		\phi& = \left( x_0 + \int_0^t \big(\frac{1}{2}z_s^2x_s-bz_sx_s+az_sx_s+v_s\big)ds+ \int_0^t \big(z_sx_s\big)dW_s, z_t \right)^{T},\\[0.2cm]
		\lambda & = \lambda_0-\lambda_1z_t,\\[0.2cm]
		\mathcal{Q} & = \dfrac{1}{w}\displaystyle\int_{\alpha\mathbb{S}}q(x,v) d v,
	\end{array}
	\right. 
\end{equation}
\normalsize
where $x_0$ refers to the initial position, $\lambda_0$ refers to the basal turning frequency of an individual cell accounting for the "spontaneous" cell motility, while the term $\lambda_1z$ represents the variation of the turning rate in response to environmental signals, we refer to \cite{conte2020glioma, hunt2018dti, meddah2023stochastic, sidani2007} for more details. Further, $w$ is a scaling constant,  given by
\begin{equation*}
	w:= \int_{\alpha \mathbb{S}} q(x,v)dv=\alpha^2.
\end{equation*}
Here, $q$ refers to the fiber distribution function. We find in the literature various expressions for this function, including the Von Mises-Fisher Distribution, the Peanut Distribution Function (PDF), and the Orientation Distribution Function (ODF) \cite{painter2013mathematical, aganj2011, conte2020glioma}. In this context, we consider the ODF defined by
\begin{equation}
	q(x,v)=\frac{1}{4\pi \mid\mathbb{D}(x)\mid^{\frac{1}{2}}(v^T(\mathbb{D}(x))^{-1}v)^{\frac{3}{2}}}\,,
	\label{q_ODF}
\end{equation}
where $\mathbb{D}(y)$ is the diffusion coefficient that accounts for information about water diffusivity in the brain, \cite{painter2013mathematical}. 
Note that in the 1D case, the diffusion tensor $\mathbb{D}(x)$ reduces to a scalar value $D(y)$, and the rest of the equation remains the same. Moreover, it is straightforward to verify that the jump rate function $\lambda$ and $\mathcal{Q}$ satisfy the global Lipschitz condition.\\
The solution of the system \eqref{Micro_sys} cannot be written in an explicit form, and thus a numerical approximation is required. Let $[0,T]$ be the time interval of interest. For a given time step $h>0$, we consider the discretisation $t_i=ih$, for $i\geq 0$, where $h=T/i$. At each discrete time point $t_i$, we denote by $\Tilde{u}_{t_i}$ the numerical approximation of the PDifMP $u_t$ given by $\Tilde{u}_{t_i}=(\Tilde{x}_{t_i}, \Tilde{z}_{t_i}, \Tilde{v}_{t_{i-1}})$.

\subsection{Numerical results}
\subsubsection{Thinned Euler-Maruyama}
We use the thinned Euler-Maruyama method, which was introduced in Section \ref{Thinning_Sec}, to simulate the dynamics of a single glioma cell at the microscopic scale. In particular, we take into account the random velocity changes that are influenced by the tumour microenvironment. Our focus is on investigating the impact of the velocity jump rate function defined as $\lambda = \lambda_0 - \lambda_1 z_t$, on the overall cellular behaviour.\\
To achieve this, we conducted several simulations, each designed to study different aspects of the proposed approach. We first specify the coefficients involved in Equation \eqref{Micro_sys}, which are reported in Table \ref{parameter_mod1}.\\

\begin{table} [H]
	\begin{center}
		\footnotesize{\begin{tabular}{c|c|c|c} 
			\hline  
			\rule{0pt}{3ex}Parameter & Description & Value (unit) & Source \\[1ex]
			\hline
			\rule{0pt}{2ex} $k^{+}$ & attachment rate    & $0.01 $ (s$^{-1}$) &\cite{lauffenburger1996receptors}\\[1.5ex]
			\rule{0pt}{2ex} $k^{-}$ & detachment rate    & $0.01 $ (s$^{-1}$) &\cite{lauffenburger1996receptors}\\[1.5ex]
			\rule{0pt}{2ex} $v$ & initial velocity of tumour cells    & $0.21 \cdot 10^{-3}$ (mm$\cdot$ s$^{-1}$) &\cite{chicoine1995}\\[1.5ex]
			\rule{0pt}{2ex} $\lambda_0$  & turning frequency  &$[0.2,5]$ (s$^{-1}$)&  \cite{sidani2007}\\[1.5ex] 
			\rule{0pt}{2ex} $\lambda_1$  & turning frequency &$[0,5]$ (s$^{-1}$)& \cite{engwer2016effective}\\[1.5ex]
            \rule{0pt}{2ex} $a$ & chemoattractant concentration  & $[0.01-3]$ & proposed value \\[1.5ex]
            \rule{0pt}{2ex} $b$ & chemorepellent concentration  & $[0.01-3]$ & proposed range \\[1.5ex]
         \rule{0pt}{2ex} $h$ & time steps  & $[10^{-4}, 10^{-2}] $ & proposed range \\[1.5ex]
         \rule{0pt}{2ex} $\lambda^{*}$ & intensity of the Poisson process  & $0.6$ & estimated \\[1.5ex]
			\hline
		\end{tabular}}
\end{center}
\vspace{0.3cm}
 \caption{\footnotesize{{\textbf{Model parameters}}}.}
\label{parameter_mod1}
\end{table}
\noindent Note that both $a$ and $b$ are dimensionless quantities, because they represent the relative strengths of the chemoattractant and repellent forces acting on the cell.\\
\begin{figure}[h!]
\centering
\includegraphics[width=1\textwidth]{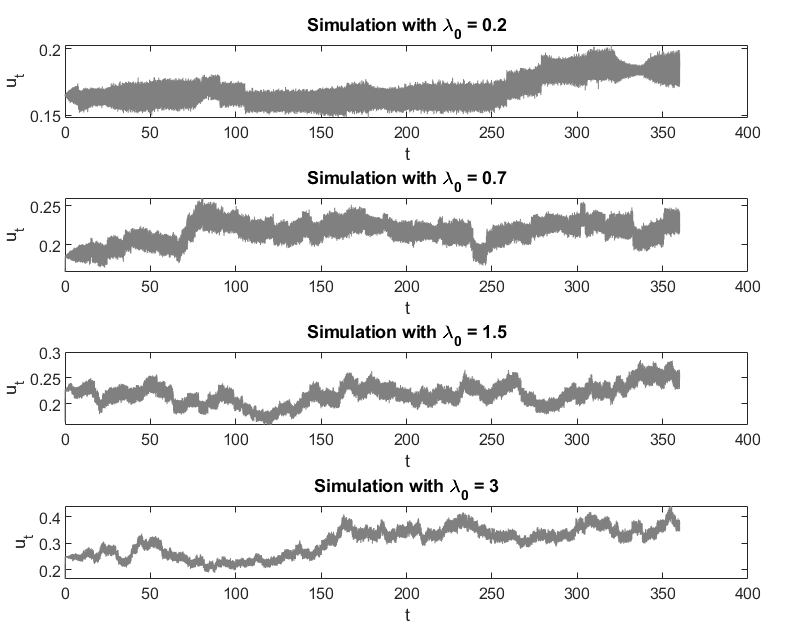}
\caption{Simulation results of the progression of glioma cell movement under different values of the basal turning rate $\lambda_0$, for fixed time steps $h=10^{-4} \, (\text{s})$ and a constant sensitivity to environmental signals parameter $\lambda_1=0.08 \, (\text{s}^{-1})$.}
\label{TEM_1}
\end{figure} 

\noindent In Figure \ref{TEM_1} and \ref{TEM_2}, we investigate the temporal evolution of $u_t$ under varying values of $\lambda_0$ and $\lambda_1$ using a time step  $h=10^{-4}$ and a simulation duration of $T=360 \,(\text{s})$. In particular, the chemoattractant concentration is set to $a=0.5$, the chemorepellent concentration is fixed at $b=0.2$, and all other model parameters are kept constant, as specified in Table \ref{parameter_mod1}.\\
Figure \ref{TEM_1} highlights the sensitivity of $u_t$ to different values of $\lambda_0$, while keeping $\lambda_1$ constant at $0.08 \,(\text{s}^{-1})$. These observations reveal a direct connection between $\lambda_0$ and the amplitude of variation observed in $u_t$.
As $\lambda_0$ decreases, the range of fluctuation is narrower, indicating a more stable and less dynamic behaviour of the glioma cell.
Conversely, higher values of $\lambda_0$ lead to larger fluctuations, suggesting enhanced cell motility and intensified interactions with the surrounding environment. This behaviour is consistent with findings from our previous study on tumour bulk dynamics, \cite{meddah2023stochastic}, which showed how varying $\lambda_0$ affects tumour spread and cellular persistence. Specifically, a lower $\lambda_0$ was associated with less frequent cellular turning, leading to more persistent migration pathways and, consequently, a greater macroscopic spread of tumour cells, particularly at the periphery of the tumour. In contrast, a higher $\lambda_0$ led to increased cellular turning, which corresponded to more dynamic and less predictable migration patterns, ultimately resulting in a reduced spread of the tumour mass.\\
\begin{figure}[h!]
\centering
\includegraphics[width=1\textwidth]{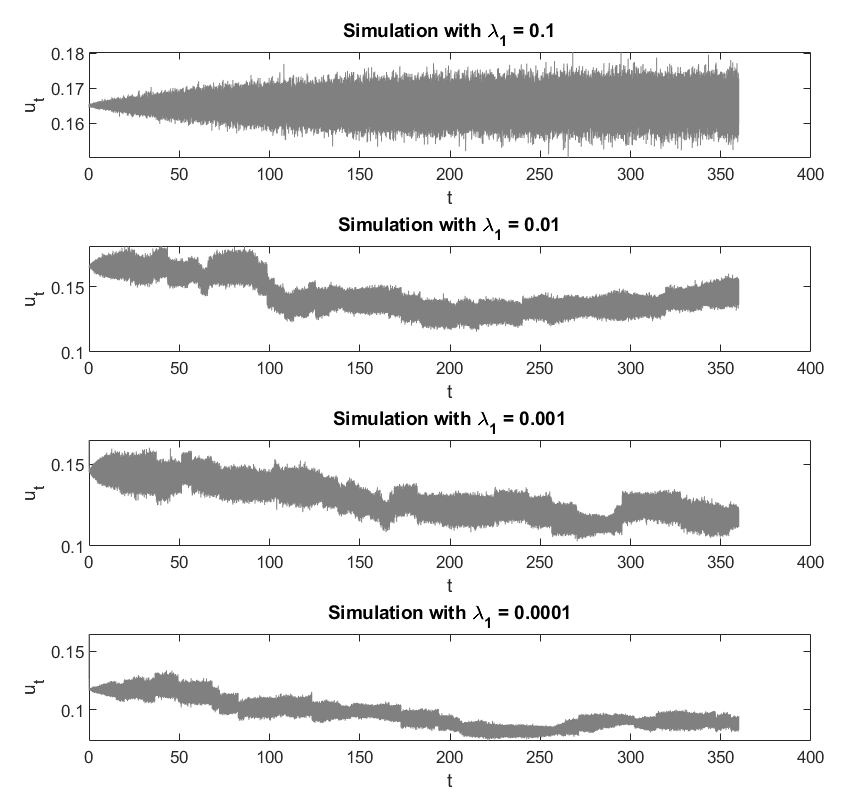}
\caption{Simulation results of the progression of glioma cell movement under different values of the basal turning rate $\lambda_0$, for fixed time steps $h=10^{-4} \, (\text{s})$ and a constant sensitivity to environmental signals parameter $\lambda_1=0.08 \, (\text{s}^{-1})$.}
\label{TEM_2}
\end{figure} 
\noindent In Figure \ref{TEM_2}, we describe the behaviour of $u_t$ under the influence of varying the values of $\lambda_1$, while $\lambda_0$ is set to $0.2 \,(\text{s}^{-1})$. As $\lambda_1$ is reduced by successive magnitudes, from $10^{-1} \,(\text{s}^{-1})$ to $10^{-4} \,(\text{s}^{-1})$, we observe a consequential decrease in the fluctuation amplitude of $u_t$. This implies a stabilising effect on glioma cell migratory behaviour. This stabilisation suggests that $\lambda_1$ is a moderator of cellular responsiveness, with higher values enhancing the ability to respond to environmental stimuli. Interestingly, the overall migration activity represented by $u_t$ shows robustness to changes in $\lambda_1$, reinforcing the idea that $\lambda_1$ primarily affects the diversity of migratory responses, rather than the average migratory speed or distance. These simulations suggest a biological interpretation where $\lambda_1$ encapsulates the adaptability of glioma cells, with lower values potentially indicating a homogenised, and perhaps less invasive, migration pattern. This echoes findings in \cite{meddah2023stochastic}, where $\lambda_1$ was critical in modulating tumour cell dynamics within brain tissue heterogeneity. Specifically, increased $\lambda_1$ enhanced tumour interaction with and adaptation to microenvironmental anisotropy, thereby influencing heterogeneity within the tumour itself. In contrast, lower values of $\lambda_1$ lead to more uniform behaviour, suggesting that cells are less influenced by the surrounding microenvironment and therefore exhibit a more homogenised, potentially less invasive migration pattern. This relationship highlights the consistency of the role of $\lambda_1$ across different scales of biological organisation and its profound influence on cellular behaviour in response to the microenvironment.

\begin{figure}[h!]
\centering
\includegraphics[width=1\textwidth]{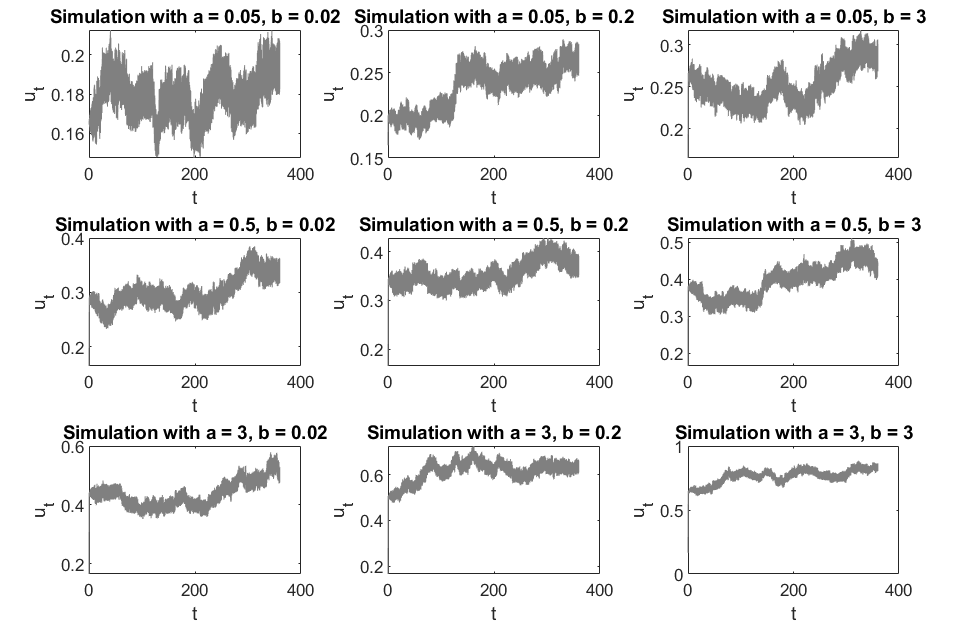}
\caption{Simulation results for the glioma cell movement for fixed time steps $h=10^{-4} \, (\text{s})$, with $\lambda_1=0.08 \, (\text{s}^{-1})$ and $\lambda_0=0.7 \, (\text{s}^{-1})$. The values of $a$ and $b$ are taken within the range in Table \ref{parameter_mod1}.}
\label{TEM_3}
\end{figure} 
\noindent Figures \ref{TEM_3} and \ref{TEM_4} illustrate the effects of varying the attractant parameter $a$ and the repellent parameter $b$ on the migratory behaviour of glioma cells, represented by the variable $u_t$ over a $360 \,(\text{s})$ time frame. Figure \ref{TEM_3} corresponds to a scenario with a larger fixed value of ${\lambda_0 = 0.7 \,(\text{s}^{-1})}$, and Figure \ref{TEM_4} represents a scenario with a smaller fixed value of ${\lambda_0 = 0.2 \,(\text{s}^{-1})}$. In both cases, the sensitivity to environmental signals parameter $\lambda_1$ is fixed to $0.08 \,(\text{s}^{-1})$.\\ 
\noindent The simulation results, presented in Figure \ref{TEM_3} with ${\lambda_0 = 0.7 \,(\text{s}^{-1})}$, shows a consistent pattern in the overall cellular motion of glioma cells as the attractant and repellent concentrations values changes. In particular, increasing $a$ from $0.05$ to $3$ while keeping $b$ low at $0.02$ (first left column), we observe a pronounced elevation in the fluctuation range of $u_t$, suggesting that higher chemoattractant concentrations significantly enhance the mobility of the glioma cell. This enhancement aligns with the chemoattractant's role as a driving force, pushing the cell towards higher attractant concentrations.
Further, when $b$ is increased to $0.2$ with a moderate $a$ value of $0.5$ (second diagonal figure), a substantial rise in $u_t$ fluctuations is also notable, indicating increased cell activity possibly due to the combined influence of attractant and repellent factors. The most dynamic cellular behaviour is observed when both attractant and repellent concentrations are simultaneously at their maximum values of 3 (third bottom right figure). The range of $u_t$ displays a remarkable expansion, reflecting the heightened motility triggered by the strong conflicting cues from the attractant and repellent. This suggests that the interplay between these external signals leads to a more complex and dynamic pattern of cellular migration.\\
\begin{figure}[h!]
\centering
\includegraphics[width=1\textwidth]{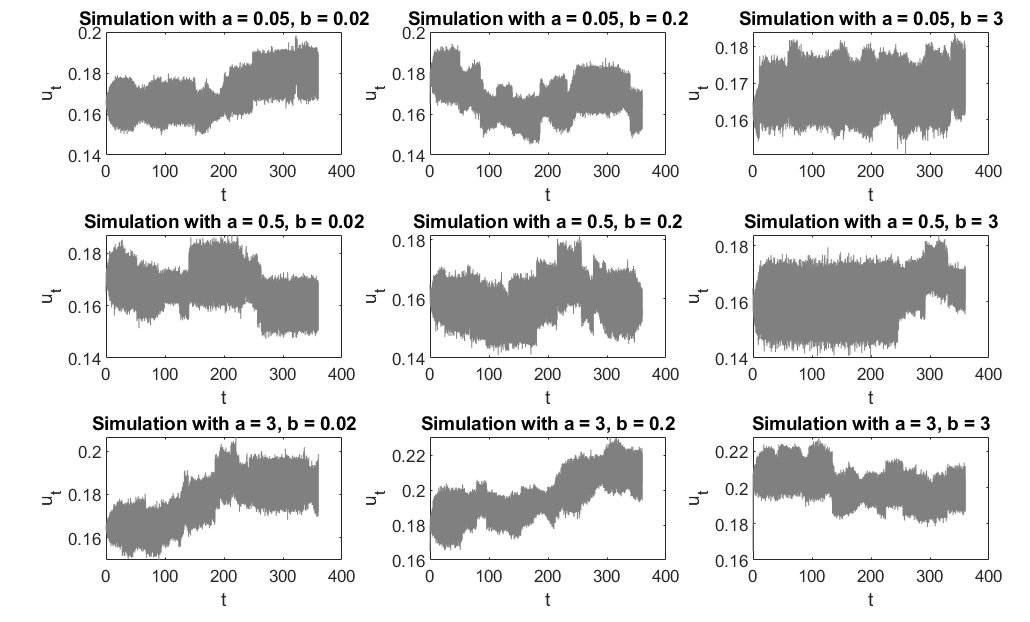}
\caption{Simulation results for the glioma cell movement for fixed time steps $h=10^{-4} \, (\text{s})$, with $\lambda_1=0.08 \, (\text{s}^{-1})$ and $\lambda_0=0.2 \, (\text{s}^{-1})$. The values of $a$ and $b$ are taken within the range in Table \ref{parameter_mod1}}
\label{TEM_4}
\end{figure} 
In the next test, we set the value of the net motility $\lambda_0$ to $0.2 \,(\text{s}^{-1})$. This lower value of $\lambda_0$ was intentionally chosen to investigate the effect of the attractant and repellent concentrations on the cellular motion while controlling the cell's intrinsic motility. By reducing $\lambda_0$, we are able to isolate the influence of external cues on the cell's migration pattern, without being overshadowed by the cell's inherent motility.
\noindent In Figure \ref{TEM_4}, the general trends remain consistent; however, the overall values of $u_t$ are lower compared to the simulations in Figure \ref{TEM_3}. This suggests that while $a$ and $b$ are primary determinants of the migratory response, the underlying motility rate $\lambda_0$ fundamentally influences the extent of this response. Lower $\lambda_0$ seems to suppress the migratory activity, which could imply that the cell's intrinsic motility is a limiting factor mediating the impact of external cues represented by $a$ and $b$.\\
\noindent These observations collectively indicate that the parameters $a$ and $b$ serve as significant controls over glioma cell migration, modulating the extent and variability of the migratory behaviour. The attractant parameter $a$ appears to promote migration, while the repellent parameter $b$ seems to inhibit it. The role of $\lambda_0$ has also been highlighted as a central factor that can enhance or attenuate cellular responses to these signals. These results provide valuable insights into the complex dynamics governing cell migration and highlight the importance of considering both internal and external factors when evaluating the migratory patterns of glioma cells.
\subsubsection{Thinned-Splitting Method for PDifMPs}

We present here a different approximation for the PDifMP called the Thinned-Splitting Method (TSM), which is a combination of two well-established simulation methods, i.e. the thinning method discussed in Section \ref{Thinning_Sec}, which efficiently simulates the jump times of the underlying jump process, combined with the splitting method, widely used to simulate the stochastic continuous dynamics of the process between jump times \cite{buckwar2022splitting, ableidinger2017stochastic, alamo2016technique}.\\
Consider a discretised time interval $[T_n,T_{n+1})$, $n\geq 1$, with equidistant time steps $h=t_m-t_{m+1}$, $m\leq n$, where $t_0=T_n$ and $t_m=T_{n+1}$. Here, $(T_n)_{n\geq 1}$ are determined as in Section \ref{Thinning_Sec}. Throughout, we denote by $(\Tilde{u}_{t_m})_{m=0,\ldots, n}$ the numerical solution of system \eqref{Micro_sys} approximating the process $(u_t)_{t\in[0,T]}$ at $t_m$, such that $\Tilde{u}(0):=u_0$.\\
The key idea behind splitting methods is to decompose the right-hand side of the system \eqref{Micro_sys} into explicitly solvable subequations and then to compose the obtained explicit solutions in a proper way.\\
\textit{Step(1):~Choice of the subequations.} We first start by rewriting the system \eqref{Micro_sys} for $u_t=(x_t,z_t,v_t)^T$ and $W_t=(W^1_t,W^2_t,W^2_t)^T$ as
\begin{equation}
\label{sys_sub_eq}
    du_t=\{f^{(1)}(u_t)dt+G(u_t) dW_t\} + f^{(2)}(u_t)dt + f^{(3)}(u_t)dt,
\end{equation}

with 
\begin{equation*}
\begin{array}{ll}
		f^{(1)}(u_t)= & \begin{pmatrix}
\frac{1}{2}z^2_tx_t-b z_tx_t+a z_tx_t +v_t\\[0.2cm]
0\\[0.2cm]
0
\end{pmatrix},\\[0.7cm]
		f^{(2)}(u_t)= & \begin{pmatrix}
0\\[0.2cm]
-(k^{+}A(x)+k^{-})z_t+f^{'}(A(x))v_tA^{'}(x)\\[0.2cm]
0
\end{pmatrix},\\[0.7cm]
		f^{(3)}(u_t)= & \begin{pmatrix}
0\\
0\\
0
\end{pmatrix},\\[0.2cm]
\end{array}
\end{equation*}
and $G(u_t)=\begin{pmatrix}
z_tx_t & 0 & 0\\
0 & 0 & 0\\
0 & 0 & 0
\end{pmatrix}$.\\

\noindent Hence, for all $t\in [0,T]$, we can rewrite the system \eqref{Micro_sys} into the following subsystems 

\begin{align}
    du^{[1]}_t &= f^{(1)}(u_t)dt+G(u_t)dW_t, \label{subsubeq:1}  \\
    du^{[2]}_t &= f^{(2)}(u_t)dt, \label{subsubeq:2}\\
    du^{[3]}_t &= f^{(3)}(u_t)dt. \label{subsubeq:3}
\end{align}

\noindent \textit{Step(2):~Exact solution of the subequations.} We denote by $\Phi^{[k]}_t$, k=1,2,3 their exact solutions at time $t$ and starting from $u_0$. Equation \eqref{subsubeq:1} is not solvable explicitly with respect to the first component of $u_t$. Therefore, we split it into two explicitly solvable differentiable equations as follows
\begin{align}
    du^{[1,1]}_t &= \frac{1}{2}z^2_tu^{[1,1]}_t dt+z_tu^{[1,1]}_tdW_t, \label{subeq:1}  \\[0.3cm]
    du^{[1,2]}_t &= (-bz_tu^{[1,2]}_t+ az_tu^{[1,2]}_t+v_t)dt, \label{subeq:2}
\end{align}
where the superscripts $[1,1]$ and $[1,2]$ correspond to the numbering of the subsystems derived from the first vector $u^{[1]}$. We denote by $\psi^{[k]}_t$, k=1,2, their exact solutions at time $t$ and starting from $x_0$ and by $\Phi^{[1]}_t$ the composition of their solutions. Note that usually the choice of the subsystems is not unique.\\
The first subequation \eqref{subeq:1} is a linear SDE with respect to $x_t$ (the first component of the PDifMP $u_t$), and has a unique explicit solution given by
\begin{equation*}
\psi^{[1]}_{h}(u^{[1,1]}_{t_{m-1}}):=u^{[1,1]}_{t_m}=u^{[1,1]}_{t_{m-1}}\exp{\left(z_{t_{m-1}}W_{t_{m-1}}\right), \qquad m=1,\ldots,n}.
    \label{solution_sub_eq1}
\end{equation*}
\noindent Further, using the fact that the ODE \eqref{subeq:2} has constant coefficients, it can be solved exactly in closed form, see \cite{chen2020structure}, with

\begin{equation*}
   \psi^{[2]}_{h}(u^{[1,2]}_{t_{m-1}}):=u^{[1,2]}_{t_{m}}=\exp{(h(-bz_{t_{m-1}}+az_{t_{m-1}}))} u^{[1,2]}_{t_{m-1}}+\frac{\exp{(h(-bz_{t_{m-1}}+az_{t_{m-1}}))}-1}{h(-bz_{t_{m-1}}+az_{t_{m-1}})}hv_{t_{m-1}}.
\end{equation*}
Therefore, using the Lie trotter composition of the flows, see \cite{mclachlan2002splitting}, the explicit solution for the first SDE \eqref{subsubeq:1} is given by
\small{\begin{align*}
    \Phi^{[1]}_h(u^{[1]}_{t_{m-1}}):= & (\psi^{[1]}_{h}\circ \psi^{[2]}_{h})(u^{[1]}_{t_{m}})\\[0.3cm]
    =& \exp{\left(z_tW_t\right)} \left( u^{[1]}_{t_{m-1}}\exp{(h(-bz_{t_{m-1}}+az_{t_{m-1}}))} +\frac{\exp{(h(-bz_t+az_t))}-1}{h(-bz_t+az_t)}hv_t)\right).
\end{align*}}
\noindent Similarly to \eqref{subeq:2}, the exact solution to Equation \eqref{subsubeq:2} is given by
\small{
\begin{align*}
     \Phi^{[2]}_{h}(u^{[2]}_{t_{m-1}})&:=u^{[2]}_{t_{m}}\\[0.3cm]
     &=\exp\left( -h(k^{+}A(x)+k^{-}) \right) u^{[2]}_{t_{m-1}}+ \FTS{\exp\left( -h(k^{+}A(x)+k^{-}) \right)-1}{-h(k^{+}A(x)+k^{-})}hf^{'}(A(x))v_{t_{m-1}}A^{'}(x)
\end{align*}}
The solution to Equation \eqref{subsubeq:3} is straightforward and given by
\begin{equation*}
    \Phi^{[3]}_{h}(u^{[3]}_{t_{m-1}}):=u^{[3]}_{t_{m}}=v_{t_{m-1}},
\end{equation*}
such that for all $m\leq n$, $v_{t_{m-1}}$ is constant over the random intervals $[T_{n-1},T_{n})$. Hence, the explicit solutions for $t\mapsto t+h$, i.e $x_t, z_t, v_t$ are given at time $t$ by

\begin{equation}
\begin{array}{ll}
		\Phi^{[1]}_h(u_t)= & \begin{pmatrix}
 \psi^{[1]}_h\big(\psi^{[2]}_t\big)\\[0.3cm]
z_t\\[0.3cm]
v_t
\end{pmatrix},\\[0.7cm]
		\Phi^{[2]}_h(u_t)= & \begin{pmatrix}
 x_t\\[0.6cm]
\exp\left( -h(k^{+}A(x)+k^{-}) \right) z_t+ \FTS{\exp\left( -h(k^{+}A(x)+k^{-}) \right)-1}{-h(k^{+}A(x)+k^{-})}hf^{'}(A(x))v_tA^{'}(x)\\[0.6cm]
v_t
\end{pmatrix}, \\[1.4cm]
\Phi^{[3]}_h(u(t))= & \begin{pmatrix}
 x_t\\
z_t\\
v_t
\end{pmatrix}.
\end{array}
\end{equation}

\noindent \textit{Step(3):~Composition of the exact solutions.} Then, for any time interval $[T_n, T_{n+1})$ with step size $h$, where ${[T_n=t_0,t_1,\ldots, t_M=T_{n+1})}$, we start with the given initial condition
$u_{T_n}=(x_{T_n},z_{T_n},v_{T_n}):=\Tilde{u}_0=(\Tilde{x}_0,\Tilde{z}_0, \Tilde{v}_0)$
and use the Lie-Trotter composition of the flows to find the numerical solution ${\Tilde{u}_1,\ldots,\Tilde{u}_m,\ldots,\Tilde{u}_M}$, such that
 \begin{equation}
     \Tilde{u}_{m+1}=\left( \Phi^{[3]}_h\circ \Phi^{([2])}_h \circ \Phi^{[1]}_h\right)(\Tilde{u}_m)=\Phi^{[3]}_h\left( \Phi^{[2]}_h\left(\Phi^{[1]}_h(\Tilde{u}_m)\right)\right).
 \end{equation}
 Therefore, we have
\begin{equation*}
\footnotesize{
\begin{aligned}
\Tilde{u}_{m+1} &= \Phi^{[3]}_h\left(\Phi^{[2]}_h \begin{pmatrix} \exp{\left(\Tilde{z}_mW_m\right)}\left(
\exp{(h(-b\Tilde{z}_m+a\Tilde{z}_m))} \Tilde{x}_m 
+ \FTS{\exp{(h(-b\Tilde{z}_m+a\Tilde{z}_m))}-1}{h(-b\Tilde{z}_m+a\Tilde{z}_m)}h\Tilde{v}_m\right) \\[0.5cm]
\Tilde{z}_m \\[0.5cm]
\Tilde{v}_m
\end{pmatrix}\right).
\end{aligned}}
\end{equation*}
Then, 
 \begin{equation*}
 \footnotesize{
\begin{array}{ll}
		\Tilde{u}_{m+1} &=\Phi^{[3]}_h \begin{pmatrix}
\exp{\left(\Tilde{z}_mW_m\right)}\left(
\exp{(h(-b\Tilde{z}_m+a\Tilde{z}_m))} \Tilde{x}_m 
+ \FTS{\exp{(h(-b\Tilde{z}_m+a\Tilde{z}_m))}-1}{h(-b\Tilde{z}_m+a\Tilde{z}_m)}h\Tilde{v}_m\right) \\[0.6cm]
\exp\left( -h(k^{+}A(\Tilde{x}_m)+k^{-}) \right) \Tilde{z}_m+ \FTS{\exp\left( -h(k^{+}A(x)+k^{-}) \right)-1}{-h(k^{+}A(\Tilde{x}_m)+k^{-})}hf^{'}(A(\Tilde{x}_m))\Tilde{v}_mA^{'}(\Tilde{x}_m)\\[0.6cm]
\Tilde{v}_m
\end{pmatrix},
\end{array}}
\end{equation*}
\noindent Finally, the solution to system \eqref{Micro_sys} is given by
 \begin{equation*}
 \footnotesize{
\begin{array}{ll}
		\Tilde{u}_{m+1}= & \begin{pmatrix}
\exp{\left(\Tilde{z}_mW_m\right)}\left(
\exp{(h(-b\Tilde{z}_m+a\Tilde{z}_m))} \Tilde{x}_m 
+ \FTS{\exp{(h(-b\Tilde{z}_m+a\Tilde{z}_m))}-1}{h(-b\Tilde{z}_m+a\Tilde{z}_m)}h\Tilde{v}_m\right) \\[0.6cm]
\exp\left( -h(k^{+}A(\Tilde{x}_m)+k^{-}) \right) \Tilde{z}_m+ \FTS{\exp\left( -h(k^{+}A(x)+k^{-}) \right)-1}{-h(k^{+}A(\Tilde{x}_m)+k^{-})}hf^{'}(A(\Tilde{x}_m))\Tilde{v}_mA^{'}(\Tilde{x}_m)\\[0.6cm]
\Tilde{v}_m
\end{pmatrix},
\end{array}}
\end{equation*}
where, $\Tilde{v}_m$ is a piecewise constant in each interval of length $t_{m+1}-t_m$.\\
\noindent Now that we have the explicit solutions, we integrate them with the thinning method. More precisely, during each simulation step, we use the thinning method to determine whether a jump occurs and its timing. If a jump is accepted, we update the state variables using the explicit solutions derived from the splitting method. \\
The detailed steps of this algorithm are provided in Algorithm \ref{algST} (see \ref{App_C}), which outlines the procedure for implementing the thinned-splitting method.

\begin{figure}[h!]
\centering
\includegraphics[width=0.85\textwidth]{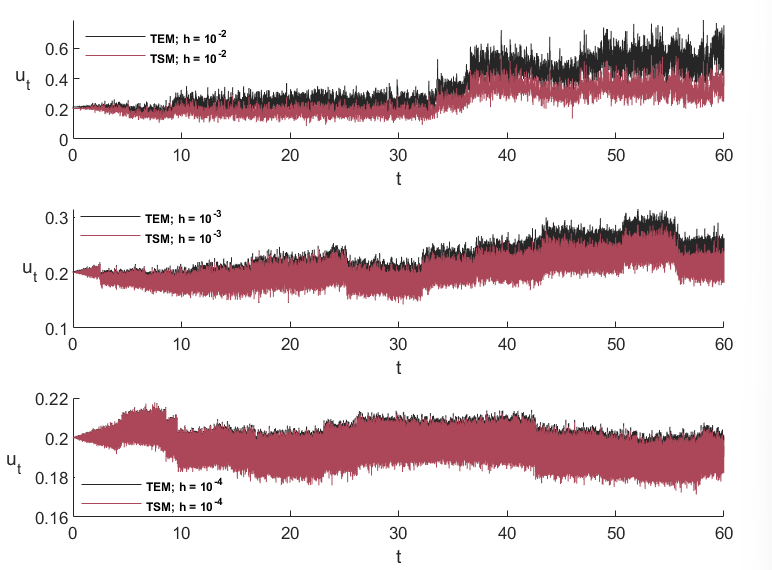}
\caption{Comparison of the glioma cell movement simulated using the TEM and TSM for different time steps $h$.  The parameter values $\lambda_1=0.08 \, (\text{s}^{-1})$, $\lambda_0=0.7 \, (\text{s}^{-1})$ were used for both methods. The values of $a$ and $b$ are taken within the range in Table \ref{parameter_mod1}}
\label{TEM_TSM}
\end{figure} 

\noindent In Figure \ref{TEM_TSM}, we present a comparative analysis of TEM and TSM applied to the simulation of the stochastic process $ u_t $ modelling the movement of glioma cells over different time steps $h$. Both methods use the same initial conditions and parameter values: $ \lambda_0 = 0.7 \, (\text{s}^{-1}) $, $ \lambda_1 = 0.08 \, (\text{s}^{-1}) $, $ a = 0.5 $, and $ b = 0.2 $.\\
The simulations are  performed with time steps $h$ of $ 0.01 $, $ 0.001 $, and $ 0.0001 $ seconds, depicted in the top, middle, and bottom subfigures, respectively. The TEM method is represented in black, while the TSM method is shown in red. The total simulation time is $ T = 60 $s.\\
For $h = 10^{-2}$, both methods exhibit fluctuations in $ u_t $, with the TEM method (black) showing a broader range of values compared to the TSM method (red). As the time step decreases to $ h = 10^{-3} $, the trajectories of both methods become more consistent with each other, displaying fewer differences. For the smallest time step, $ h = 10^{-4} $, the trajectories of the TEM and TSM methods are almost identical, indicating that both methods provide similar results when the time step is sufficiently small. This observation suggests that both methods converge as the time step decreases. However, for larger time steps, the TSM method tends to show slightly less variation, suggesting a potentially more stable behaviour under these conditions. Consequently, while both methods are effective, the TSM method may offer advantages in terms of stability for larger time steps.

\section{Discussion}
\label{section7}
The numerical simulation of PDifMPs poses a significant challenge due to their complex dynamics and the difficulty in obtaining explicit expressions for their characteristics. Although some methods have been proposed for simulating generalised stochastic hybrid systems, these methods often lack a general framework for handling PDifMPs. Additionally, they typically do not investigate the convergence properties of the employed methods, see \cite{blom2018interacting}. However, Betrazzi et al. in \cite{bertazzi2023piecewise} introduced sampling with splitting schemes for PDMPs, thereby making a significant advancement for these types of models.\\
To bridge this gap, we propose two novel numerical approaches for simulating PDifMPs; the Thinned Euler-Maruyama (TEM) method for a general framework, and the Thinned Splitting Method (TSM) for a specific model, as detailed in Section \ref{section6}. Further, we have conducted a detailed analysis of both the mean square and weak convergence of the TEM method.\\
To illustrate the performance of both the TEM and the TSM, we applied these methods to model the motion of a glioma cell at the microscopic level, and we compare the trajectories generated by each simulation. These numerical simulations allowed us to  examine the impact of microscopic parameters, such as net cell motility and environmental signals, on cellular behaviour. Our comparative analysis revealed that the TEM and TSM methods yielded nearly identical results. Notably, for larger time steps, the TSM method exhibits slightly less variation, indicating a potentially more stable behaviour under these specific conditions.\\
Looking towards future research directions, we propose further numerical analysis of both the TEM and the TSM to fully assess their efficacy and scalability. The innovative combination of thinning and splitting methods in simulating PDifMPs, particularly with the unique jump-adapted time discretisation scheme, warrants detailed exploration. Such research would significantly enhance our understanding of convergence properties in the context of PDifMPs, leading to more sophisticated and efficient numerical simulations. A focal point of our forthcoming investigations will be a comprehensive study of the TSM, with particular emphasis on its convergence properties. Additionally,  we intend to conduct  parameter estimation for the various parameters included in our model. This will provide deeper insights into the sensitivity of the model and improve the accuracy of our simulations.

\section*{Declaration of competing interest}

The authors declare that they have no known competing financial interests or personal
relationships that could have appeared to influence the work reported in this paper.

\section*{Funding}
This work was supported by the Austrian Science Fund (FWF): W1214-N15, project DK14, as well as by the strategic program ''Innovatives O\"O 2010 plus'' by the Upper Austrian Government.


\begin{appendices}
\section{}
\label{App_A}
\begin{proof}{Lemma \ref{phi_lem}}\\
Let $\phi(t,(y_1,v))$ and  $\phi(t,(y_1,v))$ satisfy the integral form (\ref{integral_form}). We have,
\begin{multline*}
     \lvert \phi(t,(y_1,v))-\phi(t,(y_2,v)) \rvert^2 = \lvert (y_1-y_2) + \int_{0}^{t}\big(b(\phi(s,(y_1,v)),v)-b(\phi(s,(y_2,v)),v)\big)ds+\\\int_{0}^{t}\big(\sigma(\phi(s,(y_1,v)),v)-\sigma(\phi(s,(y_2,v)),v)\big)dW_{s}  \rvert ^2.
\end{multline*}

\noindent Using the triangle inequality, the H\"older's inequality and the It\^o isometry, we obtain

\begin{align*}
  \lvert \phi(t,(y_1,v))-\phi(t,(y_2,v)) \rvert^2 &\leq \lvert (y_1-y_2)\rvert ^2 +\lvert \int_{0}^{t}\big(b(\phi(s,(y_1,v)),v)-b(\phi(s,(y_2,v)),v)\big)ds\rvert ^2\\
  &\qquad \qquad +\lvert \int_{0}^{t}\big(\sigma(\phi(s,(y_1,v)),v)-\sigma(\phi(s,(y_2,v)),v)\big)dW_{s}  \rvert ^2\\
  &\leq \lvert (y_1-y_2)\rvert ^2 +\int_{0}^{t}\lvert \big(b(\phi(s,(y_1,v)),v)-b(\phi(s,(y_2,v)),v)\big) \rvert ^2ds\\
  &\qquad \qquad +\int_{0}^{t}\lvert \big(\sigma(\phi(s,(y_1,v)),v)-\sigma(\phi(s,(y_2,v)),v)\big)\rvert ^2d{s}. 
\end{align*}
\noindent Using the Lipschitz continuity of the coefficients $b$ and $\sigma$ as given in (\ref{lipschitz_cond}), we get

\begin{align*}
  \lvert \phi(t,(y_1,v))-\phi(t,(y_2,v)) \rvert^2 &\leq \lvert (y_1-y_2)\rvert ^2 +K_1 \int_{0}^{t}\lvert\phi(s,(y_1,v))-\phi(s,(y_2,v))\rvert ^2ds\\
  &\qquad \qquad + K_2 \int_{0}^{t}\lvert\phi(s,(y_1,v))-\phi(s,(y_2,v))\rvert ^2ds \\
  &\leq \lvert (y_1-y_2)\rvert ^2 +2K_1 \int_{0}^{t}\lvert\phi(s,(y_1,v))-\phi(s,(y_2,v))\rvert ^2ds.
  \end{align*}
Since we start with deterministic initial values $y_1$, $y_2$, using Gronwall's inequality, we have
\footnotesize{
\begin{align*}
  \mathbb{E}  [\sup_{t\in [0,T]} \lvert \phi(t,(y_1,v))-\phi(t,(y_2,v)) \rvert^2] &\leq  \mathbb{E}  [\sup_{t\in [0,T]}\lvert y_1-y_2\rvert ^2] +2K_1 \int_{0}^{t} \mathbb{E}  [\sup_{t\in [0,T]}\lvert\phi(s,(y_1,v))-\phi(s,(y_2,v))\rvert ^2]ds.\\
  &\leq  K_3\lvert y_1-y_2\rvert ^2e^{2K_2T},\\
  &\leq \lvert y_1-y_2\rvert ^2 e^{C_1T}.
  \end{align*}}
\end{proof}

\section{}
\label{App_B}
\begin{Lemma}
Let $\mathbf{V}$ be a finite set, with $|\mathbf{V}|$ representing the cardinal number of $\mathbf{V}$ and $\kappa_i$ indicating the elements of $\mathbf{V}$ for $i = 1, 2, \ldots, |\mathbf{V}|$.
 Let $p_i$ and $\overline{p}_i$, $1 \leq i \le |\mathbf{V}|$, be two probability distributions on $\mathbf{V}$. We define the cumulative probabilities for these distributions at each index $j$ as $a_j = \sum_{i = 1}^j p_i$ and $\overline{a}_j = \sum_{i = 1}^j \overline{p}_i$, respectively, with the initial values set as $a_{0}=\overline{a}_{0}:=0$. Let $X$ and $\overline{X}$ be two $\mathbf{V}$-valued random variables defined by
\begin{equation*}
  X:=G(\mathcal{U}), \qquad \overline{X}:=\overline{G}(\mathcal{U}),  
\end{equation*}
 \noindent where
 \begin{equation*}
     \mathcal{U} \sim \mathcal{U}([0,1]), \qquad G(u)=\sum_{j=1}^{|\mathbf{V}|} \kappa_{j} \mathbbm{1}_{a_{j-1}<u \leq a_{j}} \quad \text{and} \quad \overline{G}(u)=\sum_{j=1}^{|\mathbf{V}|} \kappa_{j} \mathbbm{1}_{\overline{a}_{j-1}<u \leq \overline{a}_{j}} \quad \text{for all}\, \, {u \in[0,1]}.
 \end{equation*}
 Then, we have
\begin{equation*}
   \mathbb{P}(X \neq \overline{X}) \leq \sum_{j=1}^{|\mathbf{V}|-1}\left|a_{j}-\overline{a}_{j}\right|. 
\end{equation*}
 \label{useful_lem}
\end{Lemma}
\begin{proof}
    We refer to \cite{lemaire2020thinning} for the proof of Lemma \ref{useful_lem}.
\end{proof}

\section{}
\label{App_C}
In this appendix, we provide all the pseudo-code used in this manuscript. This includes the detailed algorithms used in our simulations and analyses.
\begin{algorithm}
\caption{PDifMP simulation using the thinning method and exact solution}
\begin{algorithmic}[1]
\State Set $T_0 = 0$, initialize $(y_0, v_0) \in E$, define a counter $n = 0$, $t = 0$, $\tau_0 = 0$.
\State Initialise homogeneous Poisson process $(N^{*}_t)_{t \in [0,T]}$ with intensity $\lambda^{*}$.
\State Define i.i.d sequences $(\mathcal{U}_k)_{k \geq 1}$ and $(\mathcal{V}_k)_{k \geq 1}$, uniformly distributed on $[0,1]$.

\While{$t < T$}
    \State Generate potential jump times $(T^{*}_k)$ from $(N^{*}_t)$.
    \State Determine the next valid jump time: $$\tau_{n+1} := \inf \{ k > \tau_n : \mathcal{U}_k \lambda^{*} \leq \lambda(\phi(T^{*}_k - T^{*}_{\tau_n}, x_n), v_n)\}.$$
    \State Set $T_{n+1} := T^{*}_{\tau_{n+1}}$ and update $t := T_{n+1}$.
    
    \State \textbf{Update the state at $T_{n+1}$}
    \[
    (y_{n+1}, v_{n+1}) = \left(\phi(T_{n+1} - T_n, x_n), \psi(\mathcal{V}_{\tau_{n+1}}, (\phi(T_{n+1} - T_n, x_n), v_n))\right)
    \]
    \State Update $x_{n+1} := (y_{n+1}, v_{n+1})$.

    \State \textbf{Continue simulating until the next jump}
    \While{$t < T_{n+1}$}
        \State Update $t$ and compute $\phi(t, x_{n+1})$.
        \State Update $v$ as required.
    \EndWhile
    
    \State $n=n+1$.
\EndWhile
\end{algorithmic}
\label{Alg_simulation_PDifMP}
\end{algorithm}

\begin{algorithm}
\caption{Pseudo code for Example 1: Simulation of the flow maps with exact and TEM Methods}
\begin{algorithmic}[1]
\State \textbf{Initialise:}
Set a counter $n=0$. Set the initial time $t_0 = 0$, the initial waiting time $\tau_0=0$ and initial values $x_0=(y_0,v_0)$ where $v_0=0$. Fix the final time $T=1$ and the time step $h$. Set $\mu$, $\sigma$, and $\lambda$.
\State Set $y_{\text{exact}} = y_0$, $y_{\text{TEM}} = y_0$, $v_t = v_0$.
\While{$t < T$}
    \State Generate the next jump time $T_{n+1}$ using exponential inter-arrival times:
    \State Draw $\tau_{n}$ from $\exp(\lambda)$ and set $T_{n+1} = T_n + \tau_{n}$
    \If{$T_{n+1} > T$}
        \State Set $T_{n+1} = T$ \Comment{exit the loop after final updates}
    \EndIf
    \State Simulate paths from $t$ to $T_{n+1}$:
    \State Initialise $t_{\text{current}} = t$
    \While{$t_{\text{current}} < T_{\text{jump}}$}
        \State Calculate $h = \min(h, (T_{n+1} - t_{\text{current}}))$ \Comment{Adjust step size}
        \State Generate $dW = \sqrt{h} \cdot \mathcal{N}(0,1)$ 
        \State Update $y_{\text{exact}}$ using the exact formula at $t_{\text{current}} + h$:
        \[
        y_{\text{exact}} = y_{{\text{exact}}_{T_n}} \cdot \exp\left(\left(\mu - 0.5 \sigma^2\right) (t_{\text{current}} + h) + \sigma \sqrt{t_{\text{current}} + h} \cdot \mathcal{N}(0,1)\right)
        \]
        \State Update $y_{\text{TEM}}$ using the Euler-Maruyama approximation:
        \[
        y_{\text{TEM}} = y_{{\text{TEM}}_{T_n}} + \mu y_{\text{TEM}} h + \sigma y_{\text{TEM}} \sqrt{h} \cdot \mathcal{N}(0,1)
        \]
        \State $t_{\text{current}} \gets t_{\text{current}} + h$
    \EndWhile
    \State Apply the jump at $T_{n+1}$:
    \State Draw $\eta$ from $\text{Exp}(\lambda)$
    \State Update both $y_{\text{exact}}$ and $y_{\text{TEM}}$ by the same factor:
    \[
    y_{\text{exact}} = y_{\text{exact}} \cdot e^\eta
    \]
    \[
    y_{\text{TEM}} = y_{\text{TEM}} \cdot e^\eta
    \]
    \State Increment the jump counter $v_t$
    \State $t \gets T_{n+1}$
    \State $n \gets n+1$
    \State Store the state ($t$, $y_{\text{exact}}$, $y_{\text{TEM}}$, $v_t$)
\EndWhile
\State Output the simulation results: trajectories of $y_{\text{exact}}$ and $y_{\text{TEM}}$, and the jump times.
\end{algorithmic}
\label{Alg_1}
\end{algorithm}

\begin{algorithm}[H]
\caption{Pseudo code for Example 2: Simulation of Geometric Brownian Motion with Adaptive Jump Rate Function}
\begin{algorithmic}[1]
\State \textbf{Initialise:}
Set a counter $n=0$. Set the initial time $t_0 = 0$, the initial waiting time $\tau_0=0$ and initial values $x_0=(y_0,v_0)$ where $v_0=0$. Fix the final time $T=1$ and the time step $h$. Set $\mu$, $\sigma$.
\State Set $y_{\text{exact}} = y_0$, $y_{\text{TEM}} = y_0$, $v_t = v_0$.
\State \textbf{Modify jump rate function:} Define $\lambda(y_t) = 0.01 \times y_t$ \Comment{Adaptive rate based on current $y_t$}
\While{$t < T$}
    \State Generate the next jump time $T_{n+1}$ considering the adaptive rate:
    \State Calculate adaptive rate $\lambda_{current} = \lambda(y_{\text{exact}})$ at $t$
    \State Draw $\tau_{n}$ from $\exp(\lambda_{current})$ and set $T_{n+1} = T_n + \tau_{n}$
    \If{$T_{n+1} > T$}
        \State Set $T_{n+1} = T$ \Comment{Exit the loop after final updates}
    \EndIf
    \State Simulate paths from $t$ to $T_{n+1}$:
    \State Initialise $t_{\text{current}} = t$
    \While{$t_{\text{current}} < T_{n+1}$}
        \State Calculate $h = \min(h, (T_{n+1} - t_{\text{current}}))$ \Comment{Adjust step size}
        \State Generate $dW = \sqrt{h} \cdot \mathcal{N}(0,1)$ 
        \State Update $y_{\text{exact}}$ using the exact formula at $t_{\text{current}} + h$:
        \[
        y_{\text{exact}} = y_{\text{exact}} \cdot \exp\left(\left(\mu - 0.5 \sigma^2\right) (t_{\text{current}} + h) + \sigma \sqrt{t_{\text{current}} + h} \cdot \mathcal{N}(0,1)\right)
        \]
        \State Update $y_{\text{TEM}}$ using the Euler-Maruyama approximation:
        \[
        y_{\text{TEM}} = y_{\text{TEM}} + \mu y_{\text{TEM}} h + \sigma y_{\text{TEM}} \sqrt{h} \cdot \mathcal{N}(0,1)
        \]
        \State $t_{\text{current}} \gets t_{\text{current}} + h$
    \EndWhile
    \State Check if jump occurs at $T_{n+1}$ using adaptive rate:
    \If{rand() $\leq \frac{\lambda(y_{\text{exact}}, T_{n+1})}{\lambda_{max}}$}
        \State Apply the jump rule: $y_{\text{exact}} = 0.9 \times y_{\text{exact}}$ 
        \State $y_{\text{TEM}} = 0.9 \times y_{\text{TEM}}$
        \State Increment the jump counter $v_t$
    \EndIf
    \State $t \gets T_{n+1}$
    \State $n \gets n+1$
    \State Store the state ($t$, $y_{\text{exact}}$, $y_{\text{TEM}}$, $v_t$)
\EndWhile
\State Output the simulation results: trajectories of $y_{\text{exact}}$ and $y_{\text{TEM}}$, and the jump times.
\end{algorithmic}
\label{Alg_2}
\end{algorithm}

\begin{algorithm}
\caption{Thinned-Splitting Method for PDifMPs}
\label{algST}
\begin{algorithmic}[1]
\State \textbf{Input:} Initial conditions $(x_0, z_0, v_0)$, time parameters $(t_{\text{First}}, t_{\text{Last}}, h)$, model parameters $(\lambda_0, \lambda_1, \lambda^{*} k^+, k^-, a, b)$
\State \textbf{Initialize:} $t \gets t_{\text{First}}$, $(x, z, v) \gets (x_0, z_0, v_0)$, $T_{\text{jumps}} \gets [t_{\text{First}}]$
\While{$t \leq t_{\text{Last}}$}
    \State $\lambda_b \gets \lambda_0 - \lambda_1 \cdot \min(z)$
    \State Sample $\tau \sim \text{Exp}(\lambda^{*})$
    \State Append $T_{\text{jumps}}$ with $T_{\text{jumps}}[\text{end}] + \tau$
    \State $t_{\text{steps}} \gets T_{\text{jumps}}[\text{end}-1] : h : T_{\text{jumps}}[\text{end}]$
    \State Sample Wiener increments $dW \sim \sqrt{h} \cdot \text{randn}(\text{size}(t_{\text{steps}}))$
    \State $\mathcal{U} \sim \text{Uniform}(0,1)$
    \If{$\mathcal{U} \leq \frac{\lambda_0 - \lambda_1 \cdot z[\text{end}]}{\lambda^{*}}$}
        \State Update $v$
        \For{$j = 1$ to $\text{length}(t_{\text{steps}})$}
            \State $x_{\text{new}} \gets \exp(z[\text{end}] \cdot dW[j]) \cdot (\exp(h \cdot (a - b) \cdot z[\text{end}] \cdot x[\text{end}]) + \frac{h \cdot v \cdot (\exp((a - b) \cdot z[\text{end}]) - 1)}{(a - b) \cdot z[\text{end}]})$
            \State $A \gets \frac{1}{1 + \exp(-x_{\text{new}})}$
            \State $z_{\text{new}} \gets z[\text{end}] \cdot \exp(-h \cdot (k^+ \cdot A + k^-)) + \frac{(1 - \exp(-h \cdot (k^+ \cdot A + k^-))) \cdot (k^+ \cdot k^-)}{(k^+ \cdot A + k^-)^3} \cdot v \cdot \frac{\exp(-x_{\text{new}})}{(1 + \exp(-x_{\text{new}}))^2}$
            \State Append $x$ with $x_{\text{new}}$
            \State Append $z$ with $z_{\text{new}}$
        \EndFor
        \State $t \gets T_{\text{jumps}}[\text{end}]$
    \Else
        \State \textbf{Reject jump:}
        \For{$j = 1$ to $\text{length}(t_{\text{steps}})$}
            \State Update $x$ and $z$ as in step 13
        \EndFor
    \EndIf
\EndWhile
\end{algorithmic}
\end{algorithm}
\end{appendices}
\newpage
\printbibliography

@article{berg1972chemotaxis,
  title={Chemotaxis in {E}scherichia {C}oli analysed by three-dimensional tracking},
  author={Berg, Howard C and Brown, Douglas A},
  journal={Nature},
  volume={239},
  number={5374},
  pages={500--504},
  year={1972},
  publisher={Nature Publishing Group}
}

@article{cloez2017probabilistic,
  title={Probabilistic and piecewise deterministic models in biology},
  author={Cloez, Bertrand and Dessalles, Renaud and Genadot, Alexandre and Malrieu, Florent and Marguet, Aline and Yvinec, Romain},
  journal={ESAIM: Proceedings and Surveys},
  volume={60},
  pages={225--245},
  year={2017},
  publisher={EDP Sciences}
}

@article{pakdaman2010fluid,
  title={Fluid limit theorems for stochastic hybrid systems with application to neuron models},
  author={Pakdaman, Khashayar and Thieullen, Michele and Wainrib, Gilles},
  journal={Advances in Applied Probability},
  volume={42},
  number={3},
  pages={761--794},
  year={2010},
  publisher={Cambridge University Press}
}

@article{buckwar2011exact,
  title={An exact stochastic hybrid model of excitable membranes including spatio-temporal evolution},
  author={Buckwar, Evelyn and Riedler, Martin G},
  journal={Journal of mathematical biology},
  volume={63},
  number={6},
  pages={1051--1093},
  year={2011},
  publisher={Springer}
}

@article{singh2010stochastic,
  title={Stochastic hybrid systems for studying biochemical processes},
  author={Singh, Abhyudai and Hespanha, Joao P},
  journal={Philosophical Transactions of the Royal Society A: Mathematical, Physical and Engineering Sciences},
  volume={368},
  number={1930},
  pages={4995--5011},
  year={2010},
  publisher={The Royal Society Publishing}
}

@article{davis1984piecewise,
  title={Piecewise-deterministic {M}arkov processes: A general class of non-diffusion stochastic models},
  author={Davis, Mark HA},
  journal={Journal of the Royal Statistical Society: Series B (Methodological)},
  volume={46},
  number={3},
  pages={353--376},
  year={1984},
  publisher={Wiley Online Library}
}

@inproceedings{blom1988piecewise,
  title={From piecewise deterministic to piecewise diffusion {M}arkov processes},
  author={Blom, HAP},
  booktitle={Proceedings of the 27th IEEE Conference on Decision and Control},
  pages={1978--1983},
  year={1988},
  organization={IEEE}
}

@inproceedings{bujorianu2003reachability,
  title={Reachability questions in piecewise deterministic {M}arkov processes},
  author={Bujorianu, Manuela L and Lygeros, John},
  booktitle={International Workshop on Hybrid Systems: Computation and Control},
  pages={126--140},
  year={2003},
  organization={Springer}
}

@article{lemaire2020thinning,
  title={Thinning and Multilevel {M}onte {C}arlo for Piecewise Deterministic ({M}arkov) Processes. {A}pplication to a stochastic {M}orris-{L}ecar model},
  author={Lemaire, Vincent and Thieullen, Mich{\`e}le and Thomas, Nicolas},
  journal={Advances in Applied Probability},
  volume={52},
  number={1},
  pages={138--172},
  year={2020},
  publisher={Cambridge University Press}
}

@inproceedings{durmus2021piecewise,
  title={Piecewise deterministic {M}arkov processes and their invariant measures},
  author={Durmus, Alain and Guillin, Arnaud and Monmarch{\'e}, Pierre},
  booktitle={Annales de l'Institut Henri Poincar{\'e}, Probabilit{\'e}s et Statistiques},
  volume={57},
  number={3},
  pages={1442--1475},
  year={2021},
  organization={Institut Henri Poincar{\'e}}
}

@article{lewis1979simulation,
  title={Simulation of nonhomogeneous {P}oisson processes by thinning},
  author={Lewis, PA W and Shedler, Gerald S},
  journal={Naval Research Logistics Quarterly},
  volume={26},
  number={3},
  pages={403--413},
  year={1979},
  publisher={Wiley Online Library}
}

@article{lemaire2018exact,
  title={Exact simulation of the jump times of a class of piecewise deterministic {M}arkov processes},
  author={Lemaire, Vincent and Thieullen, Mich{\`e}le and Thomas, Nicolas},
  journal={Journal of Scientific Computing},
  volume={75},
  number={3},
  pages={1776--1807},
  year={2018},
  publisher={Springer}
}

@article{bierkens2018piecewise,
  title={Piecewise deterministic {M}arkov processes for scalable {M}onte {C}arlo on restricted domains},
  author={Bierkens, Joris and Bouchard-C{\^o}t{\'e}, Alexandre and Doucet, Arnaud and Duncan, Andrew B and Fearnhead, Paul and Lienart, Thibaut and Roberts, Gareth and Vollmer, Sebastian J},
  journal={Statistics \& Probability Letters},
  volume={136},
  pages={148--154},
  year={2018},
  publisher={Elsevier}
}

@article{engwer2016multiscale,
  title={A multiscale model for glioma spread including cell-tissue interactions and proliferation},
  author={Engwer, Christian and Knappitsch, Markus and Surulescu, Christina},
  journal={Mathematical Biosciences \& Engineering},
  volume={13},
  number={2},
  pages={443},
  year={2016},
  publisher={American Institute of Mathematical Sciences}
}

@article{engwer2015glioma,
  title={Glioma follow white matter tracts: a multiscale {DTI}-based model},
  author={Engwer, Christian and Hillen, Thomas and Knappitsch, Markus and Surulescu, Christina},
  journal={Journal of mathematical biology},
  volume={71},
  number={3},
  pages={551--582},
  year={2015},
  publisher={Springer}
}

@article{hunt2018dti,
  title={{DTI}-based multiscale models for glioma invasion},
  author={Hunt, Alexander},
  year={2018}
}

@incollection{bujorianu2006toward,
  title={Toward a general theory of stochastic hybrid systems},
  author={Bujorianu, Manuela L and Lygeros, John},
  booktitle={Stochastic hybrid systems},
  pages={3--30},
  year={2006},
  publisher={Springer}
}

@article{chen2020structure,
  title={Structure-{P}reserving {N}umerical {I}ntegrators for {H}odgkin--{H}uxley-Type Systems},
  author={Chen, Zhengdao and Raman, Baranidharan and Stern, Ari},
  journal={SIAM Journal on Scientific Computing},
  volume={42},
  number={1},
  pages={B273--B298},
  year={2020},
  publisher={SIAM}
}

@article{riedler2013almost,
  title={Almost sure convergence of numerical approximations for piecewise deterministic {M}arkov processes},
  author={Riedler, Martin G},
  journal={Journal of Computational and Applied Mathematics},
  volume={239},
  pages={50--71},
  year={2013},
  publisher={Elsevier}
}

@book{davis2018markov,
  title={{M}arkov models and optimization},
  author={Davis, Mark HA},
  year={2018},
  publisher={Routledge}
}

@book{mao2007stochastic,
  title={Stochastic differential equations and applications},
  author={Mao, Xuerong},
  year={2007},
  publisher={Elsevier}
}

@article{meddah2023stochastic,
  title={A stochastic hierarchical model for low grade glioma evolution},
  author={Meddah, Amira and Conte, Martina and Buckwar, Evelyn},
  journal={Journal of Mathematical Biology},
  year={2023}
}

@article{kritzer2019approximation,
  title={Approximation methods for piecewise deterministic {M}arkov processes and their costs},
  author={Kritzer, Peter and Leobacher, Gunther and Sz{\"o}lgyenyi, Michaela and Thonhauser, Stefan},
  journal={Scandinavian Actuarial Journal},
  volume={2019},
  number={4},
  pages={308--335},
  year={2019},
  publisher={Taylor \& Francis}
}

@article{malhame1990jump,
  title={A jump-driven {M}arkovian electric load model},
  author={Malham{\'e}, Roland},
  journal={Advances in Applied Probability},
  volume={22},
  number={3},
  pages={564--586},
  year={1990},
  publisher={Cambridge University Press}
}

@article{ishijima2011regime,
  title={The regime switching portfolios},
  author={Ishijima, Hiroshi and Uchida, Masaki},
  journal={Asia-Pacific Financial Markets},
  volume={18},
  pages={167--189},
  year={2011},
  publisher={Springer}
}

@article{yuan2004convergence,
  title={Convergence of the {E}uler--{M}aruyama method for stochastic differential equations with {M}arkovian switching},
  author={Yuan, Chenggui and Mao, Xuerong},
  journal={Mathematics and Computers in Simulation},
  volume={64},
  number={2},
  pages={223--235},
  year={2004},
  publisher={Elsevier}
}

@article{ohgaki2009epidemiology,
  title={Epidemiology of brain tumors},
  author={Ohgaki, Hiroko},
  journal={Cancer Epidemiology: Modifiable Factors},
  pages={323--342},
  year={2009},
  publisher={Springer}
}

@phdthesis{bect2007processus,
  title={Processus de {M}arkov diffusifs par morceaux: outils analytiques et num{\'e}riques},
  author={Bect, Julien},
  year={2007},
  school={Universit{\'e} Paris Sud-Paris XI}
}

@article{baran2013feynman,
  title={Feynman-{K}ac formula for switching diffusions: connections of systems of partial differential equations and stochastic differential equations},
  author={Baran, Nicholas A and Yin, George and Zhu, Chao},
  journal={Advances in Difference Equations},
  volume={2013},
  number={1},
  pages={1--13},
  year={2013},
  publisher={SpringerOpen}
}

@book{yin2009hybrid,
  title={Hybrid switching diffusions: properties and applications},
  author={Yin, G George and Zhu, Chao},
  volume={63},
  year={2009},
  publisher={Springer Science \& Business Media}
}

@article{yin2010properties,
  title={Properties of solutions of stochastic differential equations with continuous-state-dependent switching},
  author={Yin, G and Zhu, Chao},
  journal={Journal of Differential Equations},
  volume={249},
  number={10},
  pages={2409--2439},
  year={2010},
  publisher={Elsevier}
}

@article{pages2018numerical,
  title={Numerical probability},
  author={Pag{\`e}s, Gilles},
  journal={Universitext, Springer},
  year={2018},
  publisher={Springer}
}

@book{graham2013stochastic,
  title={Stochastic simulation and Monte Carlo methods: mathematical foundations of stochastic simulation},
  author={Graham, Carl and Talay, Denis},
  volume={68},
  year={2013},
  publisher={Springer Science \& Business Media}
}

@article{talay1990expansion,
  title={Expansion of the global error for numerical schemes solving stochastic differential equations},
  author={Talay, Denis and Tubaro, Luciano},
  journal={Stochastic analysis and applications},
  volume={8},
  number={4},
  pages={483--509},
  year={1990},
  publisher={Taylor \& Francis}
}

@article{sidani2007,
  title={Cofilin determines the migration behavior and turning frequency of metastatic cancer cells},
  author={Sidani, M. and Wessels, D. and Mouneimne, G. and Ghosh, M. and Goswami, S. and Sarmiento, C. and Wang, W. and Kuhl, S. and El-Sibai, M. and Backer, J. M. and others},
  journal={The Journal of Cell Biology},
  volume={179},
  number={4},
  pages={777--791},
  year={2007},
  publisher={Rockefeller University Press},
  doi={10.1083/jcb.200707009}
}

@article{conte2020glioma,
  title={Glioma invasion and its interplay with nervous tissue and therapy: A multiscale model},
  author={Conte, Martina and Gerardo-Giorda, Luca and Groppi, Maria},
  journal={Journal of theoretical biology},
  volume={486},
  pages={110088},
  year={2020},
  publisher={Elsevier}
}

@article{engwer2016effective,
  title={Effective equations for anisotropic glioma spread with proliferation: a multiscale approach and comparisons with previous settings},
  author={Engwer, Christian and Hunt, Alexander and Surulescu, Christina},
  journal={Mathematical medicine and biology: a journal of the IMA},
  volume={33},
  number={4},
  pages={435--459},
  year={2016},
  publisher={Oxford University Press}
}

@article{chicoine1995,
  title={Assessment of brain tumor cell motility in vivo and in vitro},
  author={Chicoine, M. R. and Silbergeld, D. L.},
  journal={Journal of Neurosurgery},
  volume={82},
  number={4},
  pages={615--622},
  year={1995},
  publisher={Journal of Neurosurgery Publishing Group},
  doi={10.3171/jns.1995.82.4.0615}
}

@book{lauffenburger1996receptors,
  title={Receptors: models for binding, trafficking, and signaling},
  author={Lauffenburger, Douglas A and Linderman, Jennifer},
  year={1996},
  publisher={Oxford University Press}
}

@article{harpold2007evolution,
  title={The evolution of mathematical modeling of glioma proliferation and invasion},
  author={Harpold, Hana LP and Alvord Jr, Ellsworth C and Swanson, Kristin R},
  journal={Journal of Neuropathology \& Experimental Neurology},
  volume={66},
  number={1},
  pages={1--9},
  year={2007},
  publisher={American Association of Neuropathologists, Inc.}
}

@article{painter2013mathematical,
  title={Mathematical modelling of glioma growth: the use of diffusion tensor imaging ({DTI}) data to predict the anisotropic pathways of cancer invasion},
  author={Painter, KJ and Hillen, Thomas},
  journal={Journal of theoretical biology},
  volume={323},
  pages={25--39},
  year={2013},
  publisher={Elsevier}
}

@article{aubert2006cellular,
  title={A cellular automaton model for the migration of glioma cells},
  author={Aubert, Marine and Badoual, M and Fereol, S and Christov, C and Grammaticos, B},
  journal={Physical biology},
  volume={3},
  number={2},
  pages={93},
  year={2006},
  publisher={IOP Publishing}
}

@article{swanson2000quantitative,
  title={A quantitative model for differential motility of gliomas in grey and white matter},
  author={Swanson, Kristin R and Alvord Jr, Ellsworth C and Murray, JD},
  journal={Cell proliferation},
  volume={33},
  number={5},
  pages={317--329},
  year={2000},
  publisher={Wiley Online Library}
}

@incollection{hillen2013transport,
  title={Transport and anisotropic diffusion models for movement in oriented habitats},
  author={Hillen, Thomas and Painter, Kevin J},
  booktitle={Dispersal, individual movement and spatial ecology},
  pages={177--222},
  year={2013},
  publisher={Springer}
}

@article{hillen20042,
  title={On the $ L^{2} $-moment closure of transport equations: The {C}attaneo approximation},
  author={Hillen, Thomas},
  journal={Discrete \& Continuous Dynamical Systems-B},
  volume={4},
  number={4},
  pages={961},
  year={2004},
  publisher={American Institute of Mathematical Sciences}
}

@article{aganj2011,
  title={A {H}ough transform global probabilistic approach to multiple-subject diffusion {MRI} tractography},
  author={Aganj, Iman and Lenglet, Christophe and Jahanshad, Neda and Yacoub, Essa and Harel, Noam and Thompson, Paul M and Sapiro, Guillermo},
  journal={Medical image analysis},
  volume={15},
  number={4},
  pages={414--425},
  year={2011},
  publisher={Elsevier}
}

@article{ableidinger2017stochastic,
  title={A stochastic version of the {J}ansen and {R}it neural mass model: analysis and numerics},
  author={Ableidinger, Markus and Buckwar, Evelyn and Hinterleitner, Harald},
  journal={The Journal of Mathematical Neuroscience},
  volume={7},
  number={1},
  pages={1--35},
  year={2017},
  publisher={SpringerOpen}
}

@article{alamo2016technique,
  title={A technique for studying strong and weak local errors of splitting stochastic integrators},
  author={Alamo, Alfonso and Sanz-Serna, Jes{\'u}s Mar{\'i}a},
  journal={SIAM Journal on Numerical Analysis},
  volume={54},
  number={6},
  pages={3239--3257},
  year={2016},
  publisher={SIAM}
}

@article{buckwar2022splitting,
  title={A splitting method for {SDE}s with locally {L}ipschitz drift: Illustration on the {F}itz{H}ugh-{N}agumo model},
  author={Buckwar, Evelyn and Samson, Adeline and Tamborrino, Massimiliano and Tubikanec, Irene},
  journal={Applied Numerical Mathematics},
  volume={179},
  pages={191--220},
  year={2022},
  publisher={Elsevier}
}

@article{casas2008splitting,
  title={Splitting and composition methods in the numerical integration of differential equations},
  author={CASAS, SERGIO BLANES1 FERNANDO and MURUA, ANDER},
  journal={BOLET{\'I}N NUMERO 45 Diciembre 2008},
  pages={89},
  year={2008}
}

@article{mclachlan2002splitting,
  title={Splitting methods},
  author={McLachlan, Robert I and Quispel, G Reinout W},
  journal={Acta Numerica},
  volume={11},
  pages={341--434},
  year={2002},
  publisher={Cambridge University Press}
}

@article{petersen1998general,
  title={A general implicit splitting for stabilizing numerical simulations of {I}t{\^o} stochastic differential equations},
  author={Petersen, WP},
  journal={SIAM journal on numerical analysis},
  volume={35},
  number={4},
  pages={1439--1451},
  year={1998},
  publisher={SIAM}
}

@article{shardlow2003splitting,
  title={Splitting for dissipative particle dynamics},
  author={Shardlow, Tony},
  journal={SIAM Journal on Scientific computing},
  volume={24},
  number={4},
  pages={1267--1282},
  year={2003},
  publisher={SIAM}
}

@Article{Bertazzietal2022,
 Author = {Bertazzi, Andrea and Bierkens, Joris and Dobson, Paul},
 Title = {Approximations of piecewise deterministic {{M}arkov} processes and their convergence properties},
 FJournal = {Stochastic Processes and their Applications},
 Journal = {Stochastic Processes Appl.},
 ISSN = {0304-4149},
 Volume = {154},
 Pages = {91--153},
 Year = {2022},
 Language = {English},
 DOI = {10.1016/j.spa.2022.09.004},
 zbMATH = {7608413},
 Zbl = {1500.60044}
}

@article{bouchard2018bouncy,
  title={The bouncy particle sampler: A nonreversible rejection-free {M}arkov chain {M}onte {C}arlo method},
  author={Bouchard-C{\^o}t{\'e}, Alexandre and Vollmer, Sebastian J and Doucet, Arnaud},
  journal={Journal of the American Statistical Association},
  volume={113},
  number={522},
  pages={855--867},
  year={2018},
  publisher={Taylor \& Francis}
}

@article{bierkens2019zig,
  title={The zig-zag process and super-efficient sampling for {B}ayesian analysis of big data},
  author={Bierkens, Joris and Fearnhead, Paul and Roberts, Gareth},
  year={2019}
}

@article{platen2010jump,
  title={Jump-Adapted Strong Approximations},
  author={Platen, Eckhard and Bruti-Liberati, Nicola and Platen, Eckhard and Bruti-Liberati, Nicola},
  journal={Numerical Solution of Stochastic Differential Equations with Jumps in Finance},
  pages={347--388},
  year={2010},
  publisher={Springer}
}

@article{blom2018interacting,
  title={Interacting particle system-based estimation of reach probability for a generalized stochastic hybrid system},
  author={Blom, Henk AP and Ma, Hao and Bakker, GJ Bert},
  journal={IFAC-PapersOnLine},
  volume={51},
  number={16},
  pages={79--84},
  year={2018},
  publisher={Elsevier}
}

@article{bertazzi2023piecewise,
  title={Piecewise deterministic sampling with splitting schemes},
  author={Bertazzi, Andrea and Dobson, Paul and Monmarch{\'e}, Pierre},
  journal={arXiv preprint arXiv:2301.02537},
  year={2023}
}

@misc{Brain2021,
title = {Brain size},
	URL={https://faculty.washington.edu/chudler/ffacts.html#:~:text=The%20average%20human%20brain%20is,brain%20is%2093%20mm%20high.},	
}
\end{document}